\documentclass[11pt, letterpaper]{amsart}
\usepackage{fullpage,amsfonts,amsmath,amssymb,amsthm,tikz}
\usepackage[vcentermath]{youngtab}
\usepackage{young}
\usepackage[all,cmtip]{xy}
\usepackage{hyperref}
\newtheorem{thm}{Theorem}[section]
\newtheorem{cor}[thm]{Corollary}
\newtheorem{lem}[thm]{Lemma}
\newtheorem{prop}[thm]{Proposition}

\theoremstyle{definition}
\newtheorem{definition}{Definition}[section]
\newtheorem{example}{Example}[section]
\newtheorem*{remark}{Remark}

\begin{document}
\title{Supercharacter theories constructed by the method of little groups}
\author{Scott Andrews}
\address{Department of Mathematics \\ University of Colorado Boulder \\ Boulder, CO 80309}
\email{scott.andrews@colorado.edu}
\urladdr{http://euclid.colorado.edu/~sda/}

\begin{abstract}The \emph{method of little groups} describes the irreducible characters of semidirect products with abelian normal subgroups in terms of the irreducible characters of the factor groups. We modify this method to construct supercharacter theories of semidirect products with abelian normal subgroups. In particular, we apply this construction to reproduce known supercharacter theories of several families of unipotent groups. We also utilize our method to construct a collection of new supercharacter theories of the unipotent upper-triangular matrices.
\end{abstract}
\maketitle

\section{Introduction}

The idea of a \emph{supercharacter theory} of a finite group was introduced by Diaconis--Isaacs in \cite{MR2373317} in part to study groups whose irreducible characters are difficult to index. Loosely speaking, a supercharacter theory of a group $G$ consists of a partition of $G$ into unions of conjugacy classes along with a set of characters of $G$ which are constant on the blocks of this partition. The study of supercharacter theories began with the group $UT_n(\mathbb{F}_q)$, the group of unipotent upper-triangular matrices with entries in the $q$-element field. Classifying the irreducible representations of $G$ is known to be a `wild' problem (see \cite{MR1056208}). In separate papers, Andr\'e (\cite{MR1338979,MR1896026}), Yan (\cite{yan}), and Diaconis--Isaacs (\cite{MR2373317}) construct a supercharacter theory of $UT_n(\mathbb{F}_q)$ that is relatively easy to work with. Andr\'e--Neto produce an analogous supercharacter theory of the unipotent orthogonal and symplectic groups in \cite{MR2264135,MR2537684}. In \cite{andrews1}, the author reframes this construction in terms of group actions, allowing for it to be extended to many subgroups of the unipotent orthogonal and symplectic groups as well as to the unipotent unitary groups.

\bigbreak

Although supercharacter theories have been constructed for many groups and families of groups, there are few known general techniques for producing supercharacter theories. In \cite{MR2989654}, Hendrickson introduces several such constructions, in particular the $*$\emph{-product} of supercharacter theories. The $*$-product produces a supercharacter theory of a finite group $G$ from a $G$-invariant supercharacter theory of a normal subgroup $N$ of $G$ and a supercharacter theory of $G/N$. Unfortunately, the resulting partition of $G$ is somewhat coarse, and there is often little information to be obtained from these supercharacter theories. Recently, in \cite{lang}, Lang has classified the supercharacter theories of semidirect products where both factor groups are abelian. Many of the unipotent groups that we are interested in have nice semidirect product decompositions in which only the normal factor group is abelian. This observation motivated us to generalize the $*$-product in the case that $G = N \rtimes H$ and $N$ is abelian.

\bigbreak

If $G$ is a finite group of the form $G = M \times N$, the irreducible characters of $G$ (denoted $\text{Irr}(G)$) are indexed by pairs $(\chi,\psi)$ such that $\chi \in \text{Irr}(M)$ and $\psi \in \text{Irr}(N)$. If $G$ only admits a semidirect product decomposition, it is much harder to describe $\text{Irr}(G)$ in terms of the irreducible characters of the factor groups. The \emph{method of little groups} constructs the irreducible characters of $G = N \rtimes H$ in the case that $N$ is abelian. This technique was developed by Frobenius for finite groups and extended to topological groups by Mackey in \cite{MR515581}. Specifically, to each $\psi \in \text{Irr}(N)$ is assigned a subgroup $I_H(\psi)$ of $H$. The irreducible characters of $G$ are described in terms of pairs $(\psi,\chi)$ such that $\psi \in \text{Irr}(N)$ and $\chi \in \text{Irr}(I_H(\psi))$.

\bigbreak

In \cite{MR2774691} Marberg modifies the method of little groups to describe the supercharacters of algebra groups of the form $G = A \rtimes B$, with $A$ and $B$ algebra groups and $B$ supernormal in $G$. We take a different approach, and look more generally at any finite group of the form $G = N \rtimes H$ with $N$ abelian. Our construction assigns to each $\psi \in \text{Irr}(N)$ a subgroup $H_\psi$ of $H$, and to each $H_\psi$ a supercharacter theory. From these choices we construct a supercharacter theory of $G$ where the supercharacters are described in terms of pairs $(\psi,\chi)$ with $\psi \in \text{Irr}(N)$ and $\chi$ a supercharacter of $H_\psi$. One advantage of our method is that the subgroups $H_\psi$ can often be chosen to be nice groups with well-understood supercharacter theories. Furthermore, it is often easy to modify these supercharacter theories to obtain coarser or finer supercharacter theories, as we demonstrate in Section~\ref{seccoarsersct}.

\bigbreak

When applied to the supercharacter theory of $UT_n(\mathbb{F}_q)$, we obtain a description of the supercharacters in terms of supercharacters of nice subgroups of $UT_k(\mathbb{F}_q) \times UT_m(\mathbb{F}_q)$ for any $k$ and $m$ with $k+m=n$. In a sense, this gives an iterative construction of the supercharacter theory of $UT_n(\mathbb{F}_q)$ in terms of supercharacter theories of smaller matrix groups. In other types, our method describes the supercharacters of $UO_{2n}(\mathbb{F}_q)$, $UOSp_{2n}(\mathbb{F}_q)$, and $UU_{2n}(\mathbb{F}_q)$ (the unipotent orthogonal, symplectic, and unitary groups in even dimension) in terms of supercharacters of subgroups of $UT_n(\mathbb{F}_q)$. This gives a construction of these supercharacter theories directly from the type $A$ supercharacter theory.

\bigbreak

 We review the idea of a supercharacter theory in Section~\ref{sct}. In Section~\ref{mlg} we describe the method of little groups and develop properties of characters constructed via the method. We present our general construction in Section~\ref{sctsdp}, and use it to iteratively construct the supercharacters of $UT_n(\mathbb{F}_q)$ in Section~\ref{typea}. In Section~\ref{othertypes} we reproduce the supercharacter theories of the even-dimensional unipotent orthogonal, symplectic, and unitary groups. Finally, in Section~\ref{seccoarsersct} we construct a collection of coarser supercharacter theories of $UT_n(\mathbb{F}_q)$.

\section{Supercharacter theories}\label{sct}

The idea of a supercharacter theory of an arbitrary finite group was introduced by Diaconis--Isaacs in \cite{MR2373317}.

\begin{definition}\label{defsct}
 Let $G$ be a finite group, and suppose that $\mathcal{K}$ is a partition of $G$ into unions of conjugacy classes and $\mathcal{X}$ is a set of characters of $G$. We say that the pair $(\mathcal{K},\mathcal{X})$ is a \emph{supercharacter theory} of $G$ if
\begin{enumerate}
\item[(SCT1)] $|\mathcal{X}|=|\mathcal{K}|$,
\item[(SCT2)] the characters $\chi\in \mathcal{X}$ are constant on the blocks of $\mathcal{K}$, and
\item[(SCT3)] each irreducible character of $G$ is a constituent of exactly one character in $\mathcal{X}$.
\end{enumerate}
The characters $\chi \in \mathcal{X}$ are referred to as \emph{supercharacters} and the blocks $K \in \mathcal{K}$ are called $\emph{superclasses}$.
\end{definition}

\bigbreak

For the purposes of this paper we will be mainly concerned with the supercharacters of a supercharacter theory, and as such want a description of supercharacter theories only in terms of characters.

\bigbreak

Let
\[
        CF(G) = \{f:G \to \mathbb{C} \mid f(x) = f(yxy^{-1}) \text{ for all } x,y \in G\}
\]
be the space of class functions of the finite group $G$. There are two products on this space, the pointwise product
\[
       (fg)(x) = (f \cdot g)(x) = f(x)g(x)
\]
and the convolution product
\[
        (f * g)(x) = \frac{1}{|G|}\sum_{y \in G}f(y)g(y^{-1}x).
\]
Each product gives $CF(G)$ the structure of a commutative semisimple algebra, therefore $CF(G)$ has two natural bases of primitive idempotents (one with respect to each multiplication). The idempotents with respect to the pointwise product are the indicator functions of the conjugacy classes, and the idempotents with respect to the convolution product are the characters $\chi(1)\chi$, where $\chi$ is an irreducible character of $G$.

\bigbreak

Suppose $A \subseteq CF(G)$ is a unital subalgebra under both products; then $A$ has two natural bases of primitive idempotents. The idempotents with respect to the pointwise product are the indicator functions of the blocks of a partition of $G$ into unions of conjugacy classes, and the idempotents with respect to the convolution product are orthogonal characters of $G$. This defines a bijection
\[
        \left\{\begin{array}{l}\text{Subalgebras of } CF(G) \text{ under both the} \\ \text{pointwise and convolution products}\end{array}\right\} \longleftrightarrow \{\text{Supercharacter theories of }G\},
\]
hence an alternative characterization of a supercharacter theory. In particular, we have the following lemma.

\begin{lem}\label{ptprod} Suppose that $\mathcal{X}$ is a set of characters of $\text{Irr}(G)$ such that
\begin{enumerate}
\item the trivial character is in $\mathcal{X}$, and
\item each irreducible character of $G$ is a constituent of exactly one character of $\mathcal{X}$.
\end{enumerate}
Then there exists a partition $\mathcal{K}$ of $G$ such that the pair $(\mathcal{X},\mathcal{K})$ forms a supercharacter theory of $G$ if and only if $\mathcal{X}$ is a linear basis of a subalgebra of $\text{CF}(G)$ under the pointwise product.
\end{lem}

\begin{proof} First suppose that such a partition $\mathcal{K}$ exists. By the above remarks, the linear span of $\mathcal{X}$ is a subalgebra of $\text{CF}(G)$ under the pointwise product.

\bigbreak

Conversely, suppose that $\mathcal{X}$ is a linear basis of a subalgebra $A$ of $\text{CF}(G)$ under the pointwise product. The algebra $A$ is semisimple and commutative, hence has a basis of primitive idempotents. These idempotents must be indicator functions for disjoint subsets of $G$; as the trivial character is in $A$, these subsets partition $G$. Let $\mathcal{K}$ denote this partition of $G$; we have that
\begin{enumerate}
\item $|\mathcal{K}| = |\mathcal{X}|$, and
\item the characters in $\mathcal{X}$ are constant on the blocks of $\mathcal{K}$.
\end{enumerate}
By assumption, each irreducible character of $G$ is a constituent of exactly one character in $\mathcal{X}$.
\end{proof}

Lemma~\ref{ptprod} allows for the construction of a supercharacter theory without reference to the superclasses.

\bigbreak

Perhaps the most natural method of building a supercharacter theory of a group from supercharacter theories of subgroups is the direct product of supercharacter theories. Let $G = M \times N$, and suppose that $M$ and $N$ are each equipped with supercharacter theories, denoted $(\mathcal{K},\mathcal{X})$ and $(\mathcal{L},\mathcal{Y})$. Define
\begin{align*}
        \mathcal{M} &= \{A \times B \mid A \in \mathcal{K} \text{ and } B \in \mathcal{L}\}\text{ and} \\
        \mathcal{Z} &= \{\chi \times \psi \mid \chi \in \mathcal{X} \text{ and } \psi \in \mathcal{Y}\}.
\end{align*}

\begin{prop}[{\cite[Proposition 8.1]{MR2989654}}] The pair $(\mathcal{M},\mathcal{Z})$ defines a supercharacter theory of $G$.
\end{prop}

The supercharacter theory $(\mathcal{M},\mathcal{Z})$ is referred to as the \emph{direct product} of $(\mathcal{K},\mathcal{X})$ and $(\mathcal{L},\mathcal{Y})$.

\bigbreak

Suppose that $G$ has a normal subgroup $N$ that is equipped with a supercharacter theory $(\mathcal{K},\mathcal{X})$. We say that this supercharacter theory is $G$\emph{-invariant} if for each $g \in G$ and $n \in N$, $n$ and $gng^{-1}$ are in the same superclass. Suppose that $N$ is equipped with a $G$-invariant supercharacter theory $(\mathcal{K},\mathcal{X})$ and $G/N$ is equipped with a supercharacter theory $(\mathcal{L},\mathcal{Y})$. Define

\begin{align*}
        \mathcal{M} &= \mathcal{K} \cup \{LN\mid L \in \mathcal{L}-\{1\}\} \text{ and}\\
        \mathcal{Z} &= \{\text{Ind}_N^G(\chi) \mid \chi \in \mathcal{X}-\{1_N\}\}\cup\{\text{Inf}_{G/N}^G(\psi) \mid \psi \in \mathcal{Y}\}.
\end{align*}

\begin{prop}[{\cite[Theorem 4.3]{MR2989654}}] The pair $(\mathcal{M},\mathcal{Z})$ defines a supercharacter theory of $G$.
\end{prop}

We refer to this supercharacter theory as the $*$\emph{-product} of the supercharacter theories $(\mathcal{K},\mathcal{X})$ and $(\mathcal{L},\mathcal{Y})$, and write
\[
        (\mathcal{M},\mathcal{Z}) = (\mathcal{K},\mathcal{X})*(\mathcal{L},\mathcal{Y}).
\]

We can also construct supercharacter theories of a group from other supercharacter theories of the same group. The set of supercharacter theories of a finite group form a lattice, the properties of which are studied in depth in \cite{MR2989654}. The join of two supercharacter theories has a particularly nice description.

\begin{lem}[{\cite[Proposition 3.3]{MR2989654}}]\label{lemscljoin} Let $(\mathcal{K}_1,\mathcal{X}_1)$ and $(\mathcal{K}_2,\mathcal{X}_2)$ be supercharacter theories of a finite group $G$. Then the superclasses of $(\mathcal{K}_1,\mathcal{X}_1)\vee(\mathcal{K}_2,\mathcal{X}_2)$ correspond to the finest partition of $G$ whose blocks are unions of blocks of $\mathcal{K}_1$ and unions of blocks of $\mathcal{K}_2$.
\end{lem}

\section{The method of little groups}\label{mlg}

Clifford theory studies the relationship between the irreducible characters of a group $G$ and those of a normal subgroup $N$ of $G$. The method of little groups is one application of Clifford theory; in this section we develop some background from Clifford theory and describe the method of little groups.

\bigbreak

First we will establish some notation. If $g,h \in G$, we will often write ${}^h g = hgh^{-1}$ and $g^h = h^{-1}gh$. Similarly, if $\alpha:G \to \mathbb{C}$, the function ${}^h\alpha$ is defined by $({}^h\alpha)(g) = \alpha(g^h)$.

\bigbreak

Let $G$ be a finite group with normal subgroup $N$. If $\psi$ is a character of $N$, let $I_G(\psi)$ denote the stabilizer of $\psi$ under the conjugation action of $G$ (this is known as the \emph{inertial subgroup} of $\psi$). Lemmas~\ref{lem1}, \ref{lem2}, and \ref{extch}, along with Theorem~\ref{thmmlg} and Corollary~\ref{cormlg}, all follow from results in \cite[Section 11]{MR1038525}.

\begin{lem}\label{lem1} Let $H$ be a subgroup of $I_G(\psi)$, and let $\psi \in \textup{Irr}(N)$ be a character that can be extended to a character $\widetilde{\psi}$ of $H$. Then:
\begin{enumerate}
\item $\widetilde{\psi}$ is an irreducible character of $H$;
\item if $\chi$ is an irreducible character of $H/N$, then $\textup{Inf}_{H/N}^H(\chi)$ is an irreducible character of $H$; and
\item the set $\{\widetilde{\psi} \cdot \textup{Inf}_{H/N}^H(\chi)\mid \chi \in \textup{Irr}(H/N)\}$ is a set of distinct irreducible characters of $H$ and is independent of the choice of extension $\widetilde{\psi}$.
\end{enumerate}
\end{lem}

If $H = I_G(\psi)$, we can say even more.

\begin{lem}\label{lem2} Let $\psi \in \textup{Irr}(N)$ be a character that can be extended to a character $\widetilde{\psi}$ of $I_G(\psi)$. Then the set
\[
\{\textup{Ind}_{I_G(\psi)}^G(\widetilde{\psi} \cdot \textup{Inf}_{I_G(\psi)/N}^{I_G(\psi)}(\chi))\mid \chi \in \textup{Irr}(I_G(\psi)/N)\}
\]
 is a set of distinct irreducible characters of $G$. This set only depends on the orbit of $\psi$ under the conjugation action of $G$, and the sets of irreducible characters of $G$ associated to different conjugation orbits of irreducible characters of $N$ are distinct.
\end{lem}

If every irreducible character of $N$ can be extended to a character of $I_G(\psi)$, the above construction in fact produces all of the irreducible characters of $G$.

\begin{thm}[Method of Little Groups]\label{thmmlg}
    Let $G$ be a finite group, and let $N$ be a normal subgroup of $G$ such that each $\psi \in \textup{Irr}(N)$ can be extended to a character $\widetilde{\psi}$ of $I_G(\psi)$. Let $\mathcal{S}$ be a set of orbit representatives of $\textup{Irr}(N)$ under the conjugation action of $G$. Then the set of irreducible characters of $G$ is given by
\[
        \textup{Irr}(G) = \{\textup{Ind}_{I_G(\psi)}^G(\widetilde{\psi} \cdot \textup{Inf}_{I_G(\psi)/N}^{I_G(\psi)}(\chi)) \mid \psi \in \mathcal{S}, \chi \in \textup{Irr}(I_G(\psi)/N)\}.
\]
\end{thm}

The most common application of the method of little groups is to semidirect products. If $G = N \rtimes H$, we can identify $G/N$ with $H$. For a character $\psi \in \text{Irr}(N)$, let $I_H(\psi)$ denote the stabilizer of $\psi$ under the conjugation action of $H$. Observe that $I_H(\psi) \cong I_G(\psi)/N$. In this situation, linear characters of $N$ can always be extended to characters of $I_G(\psi)$.

\begin{lem}\label{extch} Let $G = N \rtimes H$, let $\psi$ be a linear character of $N$, and let $K \subseteq I_H(\psi)$. For $n \in N$ and $h \in K$, define $\widetilde{\psi}(nh) = \psi(n)$. Then $\widetilde{\psi}$ is an extension of $\psi$  to $NK$.
\end{lem}

A corollary of Lemma~\ref{extch} is that the method of little groups applies to semidirect products with abelian normal subgroups.

\begin{cor}\label{cormlg} Let $G = N \rtimes H$, with $N$ abelian. Let $\mathcal{S}$ be a set of orbit representatives of $\textup{Irr}(N)$ under the conjugation action of $G$. Then
\[
        \textup{Irr}(G) = \{\textup{Ind}_{I_G(\psi)}^G(\widetilde{\psi} \cdot \textup{Inf}_{I_G(\psi)/N}^{I_G(\psi)}(\chi)) \mid \psi \in \mathcal{S}, \chi \in \textup{Irr}(I_H(\psi))\}.
\]
\end{cor}

We will now establish some notation for the characters described above. Let $\psi \in \text{Irr}(N)$, $K \subseteq I_H(\psi)$, and $\widetilde{\psi}$ be the extension described in Lemma~\ref{extch}. For $\chi$ a character of $K$, define
\begin{equation}
        \psi \rtimes \chi = \textup{Ind}_{NK}^G(\widetilde{\psi} \cdot \textup{Inf}_{K}^{NK}(\chi)).
\end{equation}
The dependence on $K$ is implicit in that $\chi$ is a character of $K$.

\bigbreak

The remainder of this section will be devoted to establishing a number of properties of the characters $\psi \rtimes \chi$. It is worth mentioning that the definition of $\psi \rtimes \chi$, along with the following results, make sense for any class function $\chi$ of $K$ (and not just characters).

\begin{lem} Let $\psi$ be a linear character of $N$ and let $\chi$ be a character of $K$, with $K \subseteq I_H(\psi)$. If $n \in N$ and $h \in H$, then
\[
        (\psi \rtimes \chi)(nh) = \frac{1}{|K|} \sum_{\substack{k \in H \\ {}^k h \in K}} \psi({}^k n )\chi({}^k h).
\]
\end{lem}

\begin{proof} Let $n \in N$ and $h \in H$; then
\begin{align*}
        (\psi \rtimes \chi) (nh)
        & = \text{Ind}_{NK}^G(\widetilde{\psi} \cdot
        \text{Inf}_{K}^{NK}(\chi))(nh) \\
        & = \frac{1}{|N||K|} \sum_{\substack{m \in N,\: k \in H \\ {}^{mk}(nh) \in NK}} \psi(m({}^k n )({}^{khk^{-1}}(m^{-1})))\chi({}^k h) \\
        & = \frac{1}{|N||K|} \sum_{\substack{m \in N,\: k \in H \\ {}^k h \in K}} \psi(m)\psi({}^k n )\psi({}^{khk^{-1}}(m^{-1}))\chi({}^k h) \\
        & = \frac{1}{|K|} \sum_{\substack{k \in H \\ {}^k h \in K}} \psi({}^k n ) \chi({}^k h).
\end{align*}
\end{proof}

This formula leads to the following lemma.

\begin{lem}\label{indch} Let $K_1 \subseteq K_2 \subseteq I_H(\psi)$, and let $\chi$ be a character of $K_1$. Then
\[
        \psi \rtimes \textup{Ind}_{K_1}^{K_2}(\chi) = \psi \rtimes \chi.
\]
\end{lem}

\begin{proof} Let $n \in N$ and $h \in H$; then
\begin{align*}
        (\psi \rtimes \text{Ind}_{K_1}^{K_2}(\chi))(nh)
        & = \frac{1}{|K_2|} \sum_{\substack{k \in H \\ {}^k h \in K_2}} \psi({}^k n )\text{Ind}_{K_1}^{K_2}(\chi)({}^k h) \\
        & = \frac{1}{|K_2|} \sum_{\substack{k \in H \\ {}^k h \in K_2}} \psi({}^k n )\bigg(\frac{1}{|K_1|}\sum_{\substack{l \in K_2 \\ {}^{lk}h \in K_1}} \chi({}^{lk}h) \bigg).
\end{align*}
If $l \in K_2$ and ${}^{lk}h \in K_1$, then certainly ${}^k h \in K_2$. Furthermore, if $l \in K_2$, then $\psi({}^{lk}h) = \psi({}^k h)$. It follows that
\begin{align*}
        (\psi \rtimes \text{Ind}_{K_1}^{K_2}(\chi))(nh)
        & = \frac{1}{|K_2||K_1|} \sum_{\substack{k \in H \\ l \in K_2 \\ {}^{lk} h \in K_1}} \psi({}^k n )\chi({}^{lk}h) \\
        & = \frac{1}{|K_2||K_1|} \sum_{\substack{k \in H \\ l \in K_2 \\ {}^{lk} h \in K_1}} \psi({}^{lk} n )\chi({}^{lk}h) \\
        & = \frac{1}{|K_1|} \sum_{\substack{k \in H \\ {}^{k} h \in K_1}} \psi({}^{k} n )\chi({}^{k}h) \\
        & = (\psi \rtimes \chi)(nh).
\end{align*}
\end{proof}

For our purposes we will need to describe the products of characters in terms of the method of little groups. The following lemma is a general result about the products of induced characters.

\begin{lem}[{\cite[Theorem 10.18]{MR1038525}}]\label{prodind} Let $H$ and $K$ be subgroups of a finite group $G$, and let $\chi$ and $\psi$ be characters of $H$ and $K$, respectively. If $X$ denotes a set of $(H,K)$ double coset representatives of $G$, then
\[
        \textup{Ind}_H^G(\chi)\textup{Ind}_K^G(\psi)
        = \sum_{x \in X} \textup{Ind}_{H^x \cap K}^G (\textup{Res}_{H^x \cap K}^{H^x}(\chi^x)\textup{Res}_{H^x \cap K}^{K} (\psi)).
\]
\end{lem}

Let $G = N \rtimes H$ and let $\psi_1 \rtimes \chi_1$ and $\psi_2 \rtimes \chi_2$ be characters of $G$, where $\chi_1$ is a character of $K_1$ and $\chi_2$ is a character of $K_2$.

\begin{lem}\label{prodch} Let $X$ be a set of $(K_1,K_2)$ double coset representatives of $H$. Then
\[
        (\psi_1 \rtimes \chi_1)(\psi_2 \rtimes \chi_2) = \sum_{x \in X} (\psi_1^x\psi_2) \rtimes (\textup{Res}_{K_1^x \cap K_2}^{K_1^x}(\chi_1^x)\textup{Res}_{K_1^x \cap K_2}^{K_2}(\chi_2)).
\]
\end{lem}

\begin{proof} Note that $X$ is also a set of $(NK_1,NK_2)$ double coset representatives of $G$. By Lemma~\ref{prodind},
\begin{align*}
        &(\psi_1 \rtimes \chi_1)(\psi_2 \rtimes \chi_2) \\
        &\hspace{.25in}= \text{Ind}_{NK_1}^G(\widetilde{\psi_1} \cdot
        \text{Inf}_{K_1}^{NK_1}(\chi_1))
        \text{Ind}_{NK_2}^G(\widetilde{\psi_2} \cdot
        \text{Inf}_{K_2}^{NK_2}(\chi_2)) \\
        &\hspace{.25in}= \sum_{x \in X}
        \text{Ind}_{(NK_1)^x\cap NK_2}^G(\textup{Res}_{(NK_1)^x\cap NK_2}^{(NK_1)^x}((\widetilde{\psi_1} \cdot
        \text{Inf}_{K_1}^{NK_1}(\chi_1))^x)\textup{Res}_{(NK_1)^x\cap NK_2}^{NK_2}(\widetilde{\psi_2} \cdot
        \text{Inf}_{K_2}^{NK_2}(\chi_2))).
\end{align*}
We can rewrite the inner portion of each term in the above sum as
\begin{align*}
        &\textup{Res}_{(NK_1)^x\cap NK_2}^{(NK_1)^x}((\widetilde{\psi_1} \cdot
        \text{Inf}_{K_1}^{NK_1}(\chi_1))^x)\textup{Res}_{(NK_1)^x\cap NK_2}^{NK_2}(\widetilde{\psi_2} \cdot
        \text{Inf}_{K_2}^{NK_2}(\chi_2)) \\
        &\hspace{1in} =(\textup{Res}_{(NK_1)^x\cap NK_2}^{(NK_1)^x}(\widetilde{\psi_1}^x)\textup{Res}_{(NK_1)^x\cap NK_2}^{NK_2}(\widetilde{\psi_2})) \\
        &\hspace{1in}\cdot
        (\textup{Res}_{(NK_1)^x\cap NK_2}^{(NK_1)^x}(\text{Inf}_{K_1}^{NK_1}(\chi_1))^x\textup{Res}_{(NK_1)^x\cap NK_2}^{NK_2}
        \text{Inf}_{K_2}^{NK_2}(\chi_2)).
\end{align*}
The first term can be simplified, as
\begin{align*}
\textup{Res}_{(NK_1)^x\cap NK_2}^{(NK_1)^x}(\widetilde{\psi_1}^x)\textup{Res}_{(NK_1)^x\cap NK_2}^{NK_2}(\widetilde{\psi_2})
& = \textup{Res}_{N(K_1^x\cap K_2)}^{N(K_1)^x}(\widetilde{\psi_1^x})\textup{Res}_{N(K_1^x\cap K_2)}^{NK_2}(\widetilde{\psi_2}) \\
& = \widetilde{\psi_1^x\psi_2},
\end{align*}
where $\widetilde{\psi_1^x\psi_2}$ denotes the extension of $\psi_1^x\psi_2$ to $N(K_1^x\cap K_2)$ defined in Lemma~\ref{extch}. Furthermore, the second term can be rewritten by noting that
\begin{align*}
\textup{Res}_{(NK_1)^x\cap NK_2}^{(NK_1)^x}(\text{Inf}_{K_1}^{NK_1} (\chi_1))^x
& = \textup{Res}_{N(K_1^x\cap K_2)}^{N(K_1)^x}\text{Inf}_{K_1^x}^{N(K_1)^x} (\chi_1^x) \\
& = \text{Inf}_{K_1^x\cap K_2}^{N(K_1^x\cap K_2)}\text{Res}_{K_1^x\cap K_2}^{K_1^x}(\chi_1^x),
\end{align*}
and similarly
\[
        \textup{Res}_{(NK_1)^x\cap NK_2}^{NK_2}\text{Inf}_{K_2}^{NK_2}(\chi_2)
        = \text{Inf}_{K_1^x\cap K_2}^{N(K_1^x\cap K_2)}\text{Res}_{K_1^x\cap K_2}^{K_2}(\chi_2).
\]
This means that the second term can be simplified to
\begin{align*}
        &\textup{Res}_{(NK_1)^x\cap NK_2}^{(NK_1)^x}(\text{Inf}_{K_1}^{NK_1}(\chi_1))^x\textup{Res}_{(NK_1)^x\cap NK_2}^{NK_2}
        \text{Inf}_{K_2}^{NK_2}(\chi_2) \\
        &\hspace{1in}= \text{Inf}_{K_1^x\cap K_2}^{N(K_1^x\cap K_2)}\text{Res}_{K_1^x\cap K_2}^{K_1^x}(\chi_1^x)
        \text{Inf}_{K_1^x\cap K_2}^{N(K_1^x\cap K_2)}\text{Res}_{K_1^x\cap K_2}^{K_2}(\chi_2) \\
         &\hspace{1in}=\text{Inf}_{K_1^x\cap K_2}^{N(K_1^x\cap K_2)}(\text{Res}_{K_1^x\cap K_2}^{K_1^x}(\chi_1^x)
        \text{Res}_{K_1^x\cap K_2}^{K_2}(\chi_2)).
\end{align*}
Substituting these simplifications into the initial equation, we get that
\begin{align*}
        &(\psi_1 \rtimes \chi_1)(\psi_2 \rtimes \chi_2) \\
         &\hspace{1in}= \sum_{x \in X}
        \text{Ind}_{N(K_1^x \cap K_2)}^G (\widetilde{\psi_1^x\psi_2} \cdot \text{Inf}_{K_1^x\cap K_2}^{N(K_1^x\cap K_2)}(\text{Res}_{K_1^x\cap K_2}^{K_1^x}(\chi_1^x)
        \text{Res}_{K_1^x\cap K_2}^{K_2}(\chi_2))) \\
        &\hspace{1in}= \sum_{x \in X} (\psi_1^x\psi_2) \rtimes (\textup{Res}_{K_1^x \cap K_2}^{K_1^x}(\chi_1^x)\textup{Res}_{K_1^x \cap K_2}^{K_2}(\chi_2)).
\end{align*}
\end{proof}

\section{Supercharacter theories of semidirect products}\label{sctsdp}
In this section we present the main result of the paper and relate this result to the star product.

\subsection{Supercharacter theories constructed by the method of little groups}
Let $G = N \rtimes H$ be a finite group with $N$ abelian, and let $\mathcal{L}$ be a lattice of subgroups of $H$ (under the usual meet and join operations of the subgroup lattice) such that
\begin{enumerate}
\item[(L1)] $H$ and $\{1\}$ are in $\mathcal{L}$,
\item[(L2)] if $K \in \mathcal{L}$ and $h \in H$, then ${}^hK \in \mathcal{L}$,
\item[(L3)] each $K \in \mathcal{L}$ is equipped with a supercharacter theory, and
\item[(L4)] these supercharacter theories are compatible in the sense that induction and restriction send superclass functions to superclass functions and conjugate subgroups have conjugate supercharacter theories (more specifically, if $\chi$ is a supercharacter of $K$ then $\chi^g$ is a supercharacter of $K^g$).
\end{enumerate}
\begin{remark} The compatibility property given by (L4) is not true of many known supercharacter theories of families of groups. For example, if $G$ is an algebra group and $H$ is an algebra subgroup, both with the algebra group supercharacter theory (see Section~\ref{typea}), it is not always the case that induction sends superclass functions to superclass functions (for a counterexample see \cite[Section 6]{MR2373317}). However, if $H$ is a left (or right, or two-sided) ideal subgroup of $G$ with the left (or right, or two-sided) ideal supercharacter theory, then induction and restriction will send superclass functions to superclass functions.
\end{remark}

Suppose that for each $\psi \in \text{Irr}(N)$ we can choose a subgroup $H_\psi \in \mathcal{L}$ such that our choices of $H_\psi$ satisfy the following properties.
\begin{enumerate}
\item[(H1)] $H_\psi$ is a normal subgroup of $I_H(\psi)$.
\item[(H2)] For all $h \in H$, we have $H_{(\psi^h)} = (H_\psi)^h$.
\item[(H3)] If $1_N$ denotes the trivial character of $N$, we have $H_{1_N}=H$.
\item[(H4)] If $\psi, \varphi \in \text{Irr}(N)$, then we have $H_\psi \cap H_\varphi \subseteq H_{\psi\varphi}$.
\end{enumerate}
These conditions may seem restrictive; however, for any lattice $\mathcal{L}$ satisfying (L1)--(L4), there are maximal and minimal choices of $H_\psi$.

\begin{lem} Let $G = N \rtimes H$ be a finite group with $N$ abelian. If $\mathcal{L}$ is any lattice of subgroups of $H$ satisfying (L1)--(L4), one can always choose
\begin{enumerate}
\item $H_{1_N}=H$ and $H_\psi = \{1\}$ for all nontrivial $\psi$, or
\item $H_\psi = I_\mathcal{L}(\psi)$ (where $I_\mathcal{L}(\psi)$ is the maximal element of $\mathcal{L}$ contained in $I_H(\psi)$).
\end{enumerate}
These choices of $H_\psi$ each satisfy (H1)--(H4), and are the minimal and maximal choices of $H_\psi$, respectively.
\end{lem}

\begin{proof} The fact that these are the minimal and maximal choices of $H_\psi$ is clear, and choice (1) trivially satisfies (H1)--(H4).

\bigbreak

Let $H_\psi = I_\mathcal{L}(\psi)$ for all $\psi \in \text{Irr}(N)$ (note that such $I_\mathcal{L}(\psi)$ must exist as $\mathcal{L}$ is a lattice). As $\mathcal{L}$ is closed under conjugation, $I_\mathcal{L}(\psi)$ is normal in $I_H(\psi)$ and $I_\mathcal{L}(\psi)^h = I_\mathcal{L}(\psi^h)$ for all $h \in H$, hence we have (H1) and (H2). Condition (H3) is clear, and if $K_\psi \subseteq I_H(\psi)$ and $K_\varphi \subseteq I_H(\varphi)$, then $K_\psi \cap K_\varphi \subseteq I_H(\psi\varphi)$, hence we have (H4).
\end{proof}

\bigbreak

Given a choice of $H_\psi$, consider the set of characters
\[
        S_\psi = \{\text{Ind}_{H_\psi}^{I_H(\psi)}(\chi) \mid \chi \text{ is a supercharacter of } H_\psi\}.
\]
As $H_\psi$ is a normal subgroup of $I_H(\psi)$ and $H$-conjugates of supercharacters are also supercharacters, each irreducible character of $I_H(\psi)$ is a constituent of exactly one character in $S_\psi$ (although there may be supercharacters that induce to the same character of $I_H(\psi)$).

\bigbreak

\begin{thm}\label{sctmlg} Let $S$ be a set of orbit representatives of $\text{Irr}(N)$ under the conjugation action of $G$. Then the set of characters defined by
\[
        \textup{SCh}(\mathcal{L})=\{\psi \rtimes\chi \mid \psi \in S \text{ and } \chi \in S_\psi\}
\]
is a set of supercharacters for a supercharacter theory of $G$.
\end{thm}

\begin{proof} As each irreducible character of $I_H(\psi)$ is a constituent of exactly one character in $S_\psi$, it follows from the method of little groups that every irreducible character of $G$ is a constituent of exactly one character in $\textup{SCh}(\mathcal{L})$. Furthermore, the trivial character is in $\textup{SCh}(\mathcal{L})$. It suffices to show that the linear span of the characters in $\textup{SCh}(\mathcal{L})$ is closed under the pointwise product, and the result will follow from Lemma~\ref{ptprod}.

\bigbreak

Suppose that $\psi$ is an irreducible character of $N$ and $K \in \mathcal{L}$ is a subgroup of $H_\psi$. If $\chi$ is a character of $K$ that is constant on the superclasses, then
\[
        \psi \rtimes \chi = \psi \rtimes \text{Ind}_K^{H_\psi}(\chi).
\]
By the assumptions on $\mathcal{L}$, $\text{Ind}_K^{H_\psi}(\chi)$ is a superclass function of $H_\psi$, hence $\psi \rtimes \chi$ is a superclass function of $G$.

\bigbreak

Let $\psi_1$ and $\psi_2$ be linear characters of $N$ and let $\chi_1 \in S_{\psi_1}$ and $\chi_2 \in S_{\psi_2}$. By Lemma~\ref{prodch}, if $X$ is a set of $(H_{\psi_1},H_{\psi_2})$ double coset representatives of $G$, then
\begin{align*}
        (\psi_1 \rtimes \chi_1)(\psi_2 \rtimes \chi_2)
        &= \sum_{x \in X} (\psi_1^x\psi_2) \rtimes (\textup{Res}_{H_{\psi_1}^x \cap H_{\psi_2}}^{H_{\psi_1}^x}(\chi_1^x) \textup{Res}_{H_{\psi_1}^x \cap H_{\psi_2}}^{H_{\psi_2}}(\chi_2)) \\
        &= \sum_{x \in X} (\psi_1^x\psi_2) \rtimes (\textup{Res}_{H_{\psi_1^x} \cap H_{\psi_2}}^{H_{\psi_1^x}}(\chi_1^x) \textup{Res}_{H_{\psi_1^x} \cap H_{\psi_2}}^{H_{\psi_2}}(\chi_2)).
\end{align*}
As $\chi_1^x$ is a superclass function of $H_{\psi_1^x}$ and $\chi_2$ is a superclass function of $H_{\psi_2}$,
\[
        \textup{Res}_{H_{\psi_1^x} \cap H_{\psi_2}}^{H_{\psi_1^x}}(\chi_1^x) \textup{Res}_{H_{\psi_1^x} \cap H_{\psi_2}}^{H_{\psi_2}}(\chi_2)
\]
is a superclass function of $H_{\psi_1^x} \cap H_{\psi_2}$. Furthermore, by the assumptions on the choices of $H_\psi$ we have that $H_{\psi_1^x} \cap H_{\psi_2} \subseteq H_{\psi_1^x\psi_2}$. Hence
\[
        (\psi_1^x\psi_2) \rtimes (\textup{Res}_{H_{\psi_1^x} \cap H_{\psi_2}}^{H_{\psi_1^x}}(\chi_1^x) \textup{Res}_{H_{\psi_1^x} \cap H_{\psi_2}}^{H_{\psi_2}}(\chi_2))
\]
is a superclass function of $G$, and the linear span of the elements of $\text{SCh}(\mathcal{L})$ is closed under the pointwise product.
\end{proof}

\begin{remark} In a sense, the choice of $\mathcal{L}$ does not matter; that is, if $\mathcal{L}'$ is a different lattice of subgroups that contains all the $H_\psi$, where each $H_\psi$ has the same supercharacter theory as it does in $\mathcal{L}$, then the resulting supercharacter theory of $G$ will not change.
\end{remark}

\subsection{Relationship to the star product}

Let $G = N \rtimes H$ with $N$ abelian. Let $\text{SCT}(N)=(\mathcal{X}_N,\mathcal{K}_N)$ be the supercharacter theory of $N$ obtained from the action of the inner automorphisms of $G$, and suppose $\text{SCT}(H)=(\mathcal{X}_H,\mathcal{K}_H)$ is any supercharacter theory of $H$. The $*$-product of these supercharacter theories has supercharacter set
\[
        \{\text{Inf}_H^G(\chi) \mid \chi \in \mathcal{X}_H\} \cup \{\text{Ind}_N^G(\chi) \mid \chi \in \mathcal{X}_N-\{1_N\}\},
\]
where $1_N$ is the trivial character of $N$. This set can be written as
\[
        \{1_N \rtimes \chi \mid \chi \in \mathcal{X}_H\} \cup \{\psi \rtimes 1_{\{1\}} \mid \psi \in S\}
\]
where $S$ is a set of orbit representatives of $N$ under the conjugation action of $G$ and $1_{\{1\}}$ is the trivial character of the trivial subgroup $\{1\} \subseteq H$.

\bigbreak

Suppose that $G = N \rtimes H$ with $N$ abelian and that $\mathcal{L}$ is a sublattice of the subgroup lattice of $H$ satisfying (L1)--(L4). Given choices of $H_\psi$ for each $\psi \in \text{Irr}(N)$, we can consider two supercharacter theories of $G$. The method of little groups produces $\textup{SCh}(\mathcal{L})$; the $*$-product allows us to construct $\textup{SCT}(N)*\textup{SCT}(H)$, where $\textup{SCT}(N)$ is the supercharacter theory of $N$ obtained from the action of the inner automorphisms of $G$ and $\textup{SCT}(H)$ is the supercharacter theory of $H$ (as an element of $\mathcal{L}$).

\begin{lem} The supercharacter theory $\textup{SCh}(\mathcal{L})$ is a (not necessarily strictly) finer supercharacter theory of $G$ than $\textup{SCT}(N)*\textup{SCT}(H)$. These supercharacter theories coincide exactly when $H_\psi = \{1\}$ for all nontrivial $\psi \in \text{Irr}(N)$.
\end{lem}

\begin{proof} The characters $\{1_N \rtimes \chi \mid \chi \in \mathcal{X}_H\}$ are all supercharacters of $\textup{SCh}(\mathcal{L})$ as $H_{1_N} = H$. Furthermore, if $\psi$ is a nontrivial linear character of $N$ and $\rho$ denotes the regular character of $H_\psi$, then
\[
        \psi \rtimes 1_{\{1\}} = \psi \rtimes \text{Ind}_N^{NH_\psi}(1_{\{1\}}) = \psi \rtimes \rho.
\]
As $\rho$ is a superclass function of $H_\psi$, $\psi \rtimes 1_{\{1\}}$ is a superclass function of $G$. The final claim follows from the definition of the $*$-product.
\end{proof}

\section{Supercharacter theories of pattern groups}\label{typea}

Let $\frak{g}$ be a nilpotent algebra over $\mathbb{F}_q$. Define the \emph{algebra group} associated to $\frak{g}$ to be
\[
        G = \{1+x \mid x \in \frak{g}\},
\]
with multiplication defined by $(1+x)(1+y) = 1+(x+y+xy)$. We will often write $G = 1+\frak{g}$ to indicate that $G$ is the algebra group associated to $\frak{g}$. For example, if $\frak{ut}_n(\mathbb{F}_q)$ denotes the algebra of strictly upper-triangular matrices with entries in $\mathbb{F}_q$, then $UT_n(\mathbb{F}_q)$ is the algebra group associated to $\frak{ut}_n(\mathbb{F}_q)$.

\bigbreak

Given a poset $\mathcal{P}$ on $[n]$ that is a subset of the usual linear order, we define the \emph{pattern group} associated to $\mathcal{P}$ by
\[
        U_\mathcal{P} = \{g \in UT_n(\mathbb{F}_q) \mid g_{ij} \neq 0 \text{ implies } i \preceq_\mathcal{P} j\}.
\]
We similarly define the \emph{pattern subalgebra}
\[
        \frak{u}_\mathcal{P} = \{x \in \frak{ut}_n(\mathbb{F}_q) \mid x_{ij} \neq 0 \text{ implies } i \prec_\mathcal{P} j\};
\]
note that $U_\mathcal{P}$ is the algebra group associated to $\frak{u}_\mathcal{P}$. For more details on algebra groups and pattern groups, see \cite{MR2373317,MR2491890,MR1358482}.

\bigbreak

Let $G =1+ \frak{g}$ be an algebra group; in \cite{MR2373317}, Diaconis--Isaacs construct a supercharacter theory of $G$, which we now present. The group $G$ acts on $\frak{g}$ by left and right multiplication. Consider the bijection
\begin{align*}
        f:G &\to \frak{g}\\
        1+x &\mapsto x,
\end{align*}
and for $g \in G$, let
\[
        K_g = \{h \in G \mid f(h) \in Gf(g)G\}.
\]
The actions of $G$ on $\frak{g}$ by left and right multiplication induce actions of $G$ on the dual space $\frak{g}^*$; if $\lambda \in \frak{g}^*$, $x \in \frak{g}$ and $g,h \in G$, let
\[
        (g\lambda h) (x) = \lambda(g^{-1}xh^{-1}).
\]
Let $\theta : \mathbb{F}_q^+ \to \mathbb{C}^\times$ be a nontrivial homomorphism. For $\lambda \in \frak{g}^*$, define
\[
        \chi_\lambda = \frac{|G\lambda|}{|G\lambda G|}\sum_{\mu \in G\lambda G} \theta \circ f \circ \mu.
\]

\begin{thm}[Diaconis--Isaacs, \cite{MR2373317}]\label{alggpsct} The functions $\chi_\lambda$ are characters of $G$, and
\[
        (\{K_g \mid g \in G\},\{\chi_\lambda \mid \lambda \in \frak{g}^*\})
\]
is a supercharacter theory of $G$.
\end{thm}

\bigbreak

If $G = 1+\frak{g}$ is an algebra group, we will call a subgroup $H = 1+\frak{h}$ of $G$ a \emph{left} (respectively \emph{right}) \emph{ideal subgroup} if $\frak{h}$ a left (respectively right) ideal of $\frak{g}$. If $H = 1+\frak{h}$ is a left or right ideal subgroup of $G$, there is a coarser supercharacter theory of $H$ described below.

\begin{prop}\label{lemlissct} Let $\theta:\mathbb{F}_q^+ \to \mathbb{C}^\times$ be a nontrivial homomorphism.
\begin{enumerate}
\item If $H$ is a left ideal subgroup of $G$, there is a supercharacter theory of $H$ with superclasses defined by
\[
        K_g = \{h \in H \mid f(h) \in Gf(g)H\}
\]
and supercharacters given by
\[
        \chi_\lambda =\frac{|G\lambda|}{|G \lambda H|}\sum_{\mu \in G \lambda H} \theta \circ \mu \circ f.
\]
\item If $H$ is a right ideal subgroup of $G$, there is a supercharacter theory of $H$ with superclasses defined by
\[
        K_g = \{h \in H \mid f(h) \in Hf(g)G\}
\]
and supercharacters given by
\[
        \chi_\lambda =\frac{|\lambda G|}{|H \lambda G|}\sum_{\mu \in H \lambda G} \theta \circ \mu \circ f.
\]
\end{enumerate}
\end{prop}

In order to reproduce Theorem~\ref{alggpsct} using the method of little groups, we will first show that these subgroups satisfy (L1)--(L4). Lemmas~\ref{lemlissublattice}, \ref{lemlisindres}, and  \ref{lemlisconj} are stated in terms of left ideal subgroups but are also true of the set of right ideal subgroups.

\begin{lem}\label{lemlissublattice} The set of left ideal subgroups of $G$ is a sublattice of the subgroup lattice of $G$.
\end{lem}
\begin{proof} The only thing to check is that the join of left ideal subgroups is a left ideal subgroup. Let $\frak{a}$ and $\frak{b}$ be left ideals of $\frak{g}$; then certainly $\langle(1+\frak{a}),(1+\frak{b})\rangle \subseteq 1+(\frak{a}+\frak{b})$. Let $a \in \frak{a}$ and $b \in \frak{b}$; then
\[
        1+(a+b) = (1+a)(1+b-ab+a^2b-a^3b+\hdots),
\]
and $b-ab+a^2b-\hdots \in \frak{b}$. Thus the reverse containment holds, and the result follows.
\end{proof}

\begin{lem}\label{lemlisindres} Suppose $A\subseteq B$ are left ideal subgroups of $G$ and that $\alpha$ is a superclass function of $A$ and $\beta$ is a superclass function of $B$. Then
\begin{enumerate}
\item $\textup{Ind}_A^B(\alpha)$ is a superclass function of $B$, and
\item $\textup{Res}_A^B(\beta)$ is a superclass function of $A$.
\end{enumerate}
\end{lem}
\begin{proof} Let $b \in B$, with $b = 1+y$. Suppose $b'\in B$ is in the same superclass as $b$, and write $b' = 1+gyh$, where $g \in G$ and $h \in B$. If $k \in B$, then we have
\[
        {}^{kh}b' = 1+{}^{kh}(gyh) = 1+{}^k(hg){}^k y.
\]
It follows that ${}^k b \in A$ if and only if ${}^{kh}b' \in A$, in which case ${}^k b$ and ${}^{kh}b'$ are in the same superclass of $A$. Thus we have
\begin{align*}
        \text{Ind}_A^B(\alpha)(b)
         & = \frac{1}{|A|}\sum_{\substack{x \in B \\ {}^xb \in A}} \alpha({}^x b) \\
         & = \frac{1}{|A|}\sum_{\substack{x \in B \\ {}^{xh} b' \in A}} \alpha({}^{xh} b') \\
         & = \text{Ind}_A^B(\alpha)(b'),
\end{align*}
and $\text{Ind}_A^B(\alpha)$ is a superclass function of $B$. For (2), observe that if two elements of $A$ are in the same superclass in $A$ then they are in the same superclass in $B$.
\end{proof}
\begin{lem}\label{lemlisconj} Suppose that $H$ is a left ideal subgroup of $G$ and $\chi_\lambda$ is a supercharacter of $G$. Then for $g \in G$, we have $(\chi_\lambda)^g = \chi_{\lambda^g}$. In particular, $(\chi_\lambda)^g$ is a supercharacter of $H^g$.
\end{lem}
\begin{proof}
Note that
\begin{align*}
        \chi_\lambda
        & = \frac{|G\lambda|}{|G \lambda H|}\sum_{\mu \in G \lambda H} \theta \circ \mu \circ f \\
        & = \frac{|G\lambda|}{|G||H|}\sum_{\substack{ h \in G \\ k \in H}} \theta \circ h \lambda k \circ f,
\end{align*}
and hence we have
\begin{align*}
        \chi_\lambda^g
        & = \frac{|G\lambda|}{|G||H|}\sum_{\substack{ h \in G \\ k \in H}} \theta \circ g^{-1}h \lambda kg \circ f \\
        & = \frac{|G\lambda^g|}{|G||H|}\sum_{\substack{ h \in G \\ k \in H}} \theta \circ h^g \lambda^g k^g \circ f \\
        & = \chi_{\lambda^g}.
\end{align*}
\end{proof}

Let $G = 1+\frak{g}$ be a pattern subgroup of $UT_n(\mathbb{F}_q)$. In order to reconstruct the algebra group supercharacter theory of $G$ using Theorem~\ref{sctmlg}, we need to decompose $G$ as a semidirect product $N \rtimes H$ with $N$ abelian. Let $n = k+m$, and define
\begin{align*}
        H_k &= \{g \in G \mid g_{ij} = 0 \text{ if } j>k \text{ and } j>i\}, \\
        H_m &= \{g \in G \mid g_{ij} = 0 \text{ if } i\leq k \text{ and } i<j\}, \text{ and} \\
        N &= \{g \in G \mid g_{ij} = 0 \text{ if } i<j \text{ and } j \leq k \text{ or } i>k\}.
\end{align*}
\begin{example} Let $n=7$, $k=3$, and $m=4$, and $G = UT_7(\mathbb{F}_q)$; we have
\begin{align*}
        H_k & = \left\{\left(\begin{array}{lll|llll}
        1 & * & * & 0 & 0 & 0 & 0 \\
        0 & 1 & * & 0 & 0 & 0 & 0 \\
        0 & 0 & 1 & 0 & 0 & 0 & 0 \\ \hline
        0 & 0 & 0 & 1 & 0 & 0 & 0 \\
        0 & 0 & 0 & 0 & 1 & 0 & 0 \\
        0 & 0 & 0 & 0 & 0 & 1 & 0 \\
        0 & 0 & 0 & 0 & 0 & 0 & 1
        \end{array}\right) \in UT_7(\mathbb{F}_q)\right\},\\
        H_m & = \left\{\left(\begin{array}{lll|llll}
        1 & 0 & 0 & 0 & 0 & 0 & 0 \\
        0 & 1 & 0 & 0 & 0 & 0 & 0 \\
        0 & 0 & 1 & 0 & 0 & 0 & 0 \\ \hline
        0 & 0 & 0 & 1 & * & * & * \\
        0 & 0 & 0 & 0 & 1 & * & * \\
        0 & 0 & 0 & 0 & 0 & 1 & * \\
        0 & 0 & 0 & 0 & 0 & 0 & 1
        \end{array}\right) \in UT_7(\mathbb{F}_q)\right\},\text{ and} \\
        N & = \left\{\left(\begin{array}{lll|llll}
        1 & 0 & 0 & * & * & * & * \\
        0 & 1 & 0 & * & * & * & * \\
        0 & 0 & 1 & * & * & * & * \\ \hline
        0 & 0 & 0 & 1 & 0 & 0 & 0 \\
        0 & 0 & 0 & 0 & 1 & 0 & 0 \\
        0 & 0 & 0 & 0 & 0 & 1 & 0 \\
        0 & 0 & 0 & 0 & 0 & 0 & 1
        \end{array}\right) \in UT_7(\mathbb{F}_q)\right\}.
\end{align*}
\end{example}

Observe that
\begin{enumerate}
\item $H_k$ is isomorphic to a pattern subgroup of $UT_{k}(\mathbb{F}_q)$ and $H_m$ is isomorphic to a pattern subgroup of $UT_{m}(\mathbb{F}_q)$;
\item $N$ is abelian; and
\item if $H = H_k \times H_m$, then $G = N \rtimes H$.
\end{enumerate}

Let $\mathcal{L}$ be the lattice of subgroups of $H_k \times H_m$ of the form $K_k \times K_m$, with $K_k$ a right ideal subgroup of $H_k$ and $K_m$ a left ideal subgroup of $H_m$. Equip $K_k \times K_m$ with the direct product of the right ideal supercharacter theory of $K_k$ and the left ideal supercharacter theory of $K_m$.

\bigbreak

We will denote by $\frak{h},\frak{h}_k,\frak{h_m}$ and $\frak{n}$ the algebras $f(H),f(H_k)$, $f(H_m)$, and $f(N)$. We have that
\begin{enumerate}
\item the irreducible characters of $N$ are exactly the elements of the form $\psi_\lambda =\theta \circ \lambda \circ f$, where $\lambda \in \frak{n}^*$; and
\item for $h = (h_k,h_m) \in H$ and $x \in \frak{n}$, ${}^h(1+x) = 1 +h_k x h_m^{-1}$.
\end{enumerate}
For each $\psi_\lambda \in \text{Irr}(N)$, define $H_{\psi_\lambda} = I_\mathcal{L}(\psi_\lambda)$.

\begin{example} Once again let $n=7$, $k=3$, and $m=4$, and let $G = UT_7(\mathbb{F}_q)$. Define $\lambda \in \frak{n}^*$ by
\[
        \lambda(x) = x_{15}+x_{37}.
\]
In other words, $\lambda(x)$ is determined by the entries denoted by $\circ$ below:
\[
        \left(\begin{array}{lll|llll}
        1 & 0 & 0 & * & \circ & * & * \\
        0 & 1 & 0 & * & * & * & * \\
        0 & 0 & 1 & * & * & * & \circ \\ \hline
        0 & 0 & 0 & 1 & 0 & 0 & 0 \\
        0 & 0 & 0 & 0 & 1 & 0 & 0 \\
        0 & 0 & 0 & 0 & 0 & 1 & 0 \\
        0 & 0 & 0 & 0 & 0 & 0 & 1
        \end{array}\right).
\]
For this choice of $\lambda$, $H_{\psi_\lambda}$ will be the subgroup
\[
H_{\psi_\lambda}  = \left\{\left(\begin{array}{lll|llll}
        1 & 0 & 0 & 0 & \circ & 0 & 0 \\
        0 & 1 & * & 0 & 0 & 0 & 0 \\
        0 & 0 & 1 & 0 & 0 & 0 & \circ \\ \hline
        0 & 0 & 0 & 1 & 0 & * & 0 \\
        0 & 0 & 0 & 0 & 1 & * & 0 \\
        0 & 0 & 0 & 0 & 0 & 1 & 0 \\
        0 & 0 & 0 & 0 & 0 & 0 & 1
        \end{array}\right) \in H_k \times H_m\right\},
\]
where once again the entries labeled with $\circ$ are the entries which determine $\lambda$. For comparison, note that $I_H(\psi_\lambda)$ is the subgroup
\[
I_H(\psi_\lambda)  = \left\{\left(\begin{array}{lll|llll}
        1 & 0 & a & 0 & \circ & 0 & 0 \\
        0 & 1 & * & 0 & 0 & 0 & 0 \\
        0 & 0 & 1 & 0 & 0 & 0 & \circ \\ \hline
        0 & 0 & 0 & 1 & 0 & * & 0 \\
        0 & 0 & 0 & 0 & 1 & * & a \\
        0 & 0 & 0 & 0 & 0 & 1 & 0 \\
        0 & 0 & 0 & 0 & 0 & 0 & 1
        \end{array}\right) \in H_k \times H_m\right\}
\]
where $a$ ranges over the elements of $\mathbb{F}_q$. Note that $I_H(\psi_\lambda)$ is an algebra group but not a pattern group.
\end{example}

\begin{thm}\label{thmmlgtypea} For each $\psi_\lambda \in \text{Irr}(N)$, define $H_{\psi_\lambda} = I_\mathcal{L}(\psi_\lambda)$. These choices of $H_{\psi_\lambda}$ produce the same supercharacter theory of $G$ as described in Theorem~\ref{alggpsct}.
\end{thm}

\begin{proof}
Fix an element $\eta \in \frak{g}^*$. Let $\lambda = \eta|_{\frak{n}}$, and define
\begin{align*}
        K_k &= \{h \in H_k \mid h\lambda = \lambda\} \text{ and} \\
        K_m &= \{h \in H_m \mid \lambda h = \lambda\}.
\end{align*}
Note that $K_k$ is a right ideal subgroup of $H_k$, $K_m$ is a left ideal subgroup of $H_m$, and
\[
        K = K_k \times K_m = I_\mathcal{L}(\psi_\lambda).
\]
Denote $f(K),f(K_k)$ and $f(K_m)$ by $\frak{k},\frak{k}_k$ and $\frak{k}_m$.

\bigbreak

Let $\mu = \eta|_{\frak{k}}$; note that $\mu$ can be written uniquely as $\mu_k \oplus \mu_m$, where $\mu_k \in \frak{k}_k^*$ and $\mu_m \in \frak{k}_m^*$. Let
\begin{align*}
        \frak{r}_\mu &= \{x \in \frak{k}_k \mid \mu_k(xy) = 0 \text{ for all }y \in \frak{h}_k\} \text{ and} \\
       \frak{l}_\mu &= \{x \in \frak{k}_m \mid \mu_k(yx) = 0 \text{ for all }y \in \frak{h}_m\},
\end{align*}
and denote the corresponding algebra groups by $R_\mu$ and $L_\mu$. Then the supercharacters of $K$ under the direct product of the right ideal supercharacter theory of $K_k$ and the left ideal supercharacter theory of $K_m$ are the elements
\[
        \chi_\mu = \text{Ind}_{R_\mu \times L_\mu}^K \text{Res}_{R_\mu \times L_\mu}^K(\theta \circ \mu \circ f).
\]

Observe that
\begin{align*}
        \psi_\lambda \rtimes \chi_\mu
        & = \psi_\lambda \rtimes \text{Res}_{R_\mu \times L_\mu}^K(\theta \circ \mu \circ f) \\
        & = \text{Ind}_{N(R_\mu \times L_\mu)}^G(\theta \circ (\mu |_{\frak{r}_\mu \oplus\frak{l}_\mu}\oplus \lambda)\circ f).
\end{align*}
We also have
\[
        \frak{n}\oplus\frak{r}_\mu \oplus\frak{l}_\mu = \{x \in \frak{g} \mid g\eta h (x) = \eta(x) \text{ for all } g\in NH_m,h \in NH_k\},
\]
and it follows that
\[
        \{\nu \in \frak{g}^* \mid \nu|_{\frak{n}\oplus\frak{r}_\mu \oplus\frak{l}_\mu} = \eta|_{\frak{n}\oplus\frak{r}_\mu \oplus\frak{l}_\mu}\} = NH_m\eta NH_k.
\]

We now have
\begin{align*}
        \text{Ind}_{N(R_\mu \times L_\mu)}^G(\theta \circ (\mu |_{\frak{r}_\mu \oplus\frak{l}_\mu}\oplus \lambda)\circ f)
        & = \frac{1}{|G|}\sum_{g \in G} \sum_{\substack{\nu \in \frak{g}^* \\ \nu|_{\frak{n}\oplus\frak{r}_\mu \oplus\frak{l}_\mu} = \eta|_{\frak{n}\oplus \frak{r}_\mu \oplus\frak{l}_\mu}}} \theta \circ g\nu g^{-1} \circ f \\
        & = \frac{1}{|G|}\sum_{g \in G} \sum_{\nu \in NH_m\eta NH_k} \theta \circ g\nu g^{-1} \circ f \\
        & = \frac{|NH_m\eta NH_k|}{|G||H||N|^2}\sum_{g \in G} \sum_{\substack{(h_k,h_m) \in H \\ a,b \in N}} \theta \circ g (ah_k\eta bh_m )g^{-1} \circ f \\
        & = \frac{|NH_m\eta NH_k|}{|G|^2}\sum_{g_1,g_2 \in G} \theta \circ g_1\eta g_2 \circ f \\
        & = \frac{|NH_m\eta NH_k|}{|G \eta G|}\sum_{\nu \in G \eta G} \theta \circ \nu  \circ f.
\end{align*}
Note that $NH_m \times NH_k$ is a normal subgroup of $G\times G$, hence orbit representative choice does not affect $|NH_m\eta NH_k|$.
\end{proof}

Our construction yields supercharacters which might differ from those defined in Theorem~\ref{alggpsct} by a constant multiple, depending on whether $|NH_m\eta NH_k|$ and $|G\eta|$ are equal. In the case that $G = UT_n(\mathbb{F}_q)$, we in fact have that $|NH_m\eta NH_k|=|G\eta|$, although in general it seems to be unknown whether these orbits are of equal size.

\section{Supercharacter theories of unipotent groups of other types}\label{othertypes}

In \cite{andrews1}, the author constructs supercharacter theories of a large collection of unipotent groups which includes the unipotent orthogonal, symplectic, and unitary groups. This construction generalizes supercharacter theories constructed by Andr\'e--Neto in \cite{MR2457229,MR2264135,MR2537684}. In this section we review this construction and reproduce it using the method of little groups.

\subsection{Supercharacter theories of unipotent groups defined by anti-involutions}

Let $q$ a power of an odd prime and let $G$ be a pattern subgroup of $UT_{2n}(\mathbb{F}_{q^k})$ for some $n$ and $k$. Let $\frak{g}$ be the corresponding subalgebra of $\frak{ut}_{2n}(\mathbb{F}_{q^k})$, considered as an $\mathbb{F}_q$-algebra. Let $\dagger:\frak{g} \to \frak{g}$ be an algebra anti-involution such that $(\alpha e_{ij})^\dagger \in \mathbb{F}_{q^k}^\times e_{\bar{j}\bar{i}}$ for all $\alpha \in \mathbb{F}_{q^k}^\times$ and $i<j$ (where $\bar{i} = n+1-i$). For $g =1+x \in G$, define $g^\dagger = 1+x^\dagger$; this gives an anti-involution of $G$. Consider the group
\[
        U = \{u \in G \mid u^\dagger = u^{-1}\},
\]
along with the corresponding Lie algebra
\[
        \frak{u} = \{x \in \frak{g} \mid x^\dagger = -x\}.
\]
There are actions of $G$ on $\frak{u}$ and $\frak{u}^*$ defined by
\begin{align*}
    g \cdot x & = gxg^\dagger \text{ and} \\
    (g \cdot \lambda)(x)& = \lambda(g^{-1}xg^{-\dagger})
\end{align*}
for $g \in G$, $x \in \frak{u}$, and $\lambda \in \frak{u}^*$.

\bigbreak

Let
\[
\frak{h} = \bigg\{ x\in \frak{g} \:\bigg|\: x_{ij} = 0 \text{ if } j\leq\frac{n}{2}\bigg\},
\]
and define $H = \frak{h}+1$.

\begin{thm}[{\cite[Theorem 6.1]{andrews1}}]\label{sctofu} Let $f$ be any Springer morphism (as in \cite{andrews1}) and let $\theta:\mathbb{F}_q^+ \to \mathbb{C}^\times$ be a nontrivial homomorphism. For $u \in U$ and $\lambda \in \frak{u}^*$, define
\[
        K_u = \{ v \in U \mid f(v) \in G \cdot f(u)\}\quad \text{and}\quad
        \chi_\lambda = \frac{|H \cdot \lambda|}{|G \cdot \lambda|} \sum_{\mu \in G \cdot \lambda}
        \theta \circ \mu \circ f.
\]
We have the following.
\begin{enumerate}
\item The functions $\chi_\lambda$ are characters of $U$.
\item The partition of $U$ given by $\mathcal{K} = \{K_u \mid u \in U\}$, along with $\mathcal{X}=\{\chi_\lambda \mid \lambda \in \frak{u}^*\}$, form a supercharacter theory of $U$. This supercharacter theory is independent of the choice of $\theta$ and $f$.
\end{enumerate}
\end{thm}

In order to reproduce this supercharacter theory via the method of little groups, we need to decompose $U$ as a semidirect product. For each $h \in UT_n(\mathbb{F}_{q^k})$,
define $\widetilde{h} \in UT_n(\mathbb{F}_{q^k})$ by
\[
        \left(\begin{array}{ll} I & 0 \\ 0 & \widetilde{h} \end{array}\right) = \left(\begin{array}{ll} h & 0 \\ 0 & I \end{array}\right)^\dagger.
\]
Note that if
\[
        \left(\begin{array}{ll} h & 0 \\ 0 & I \end{array}\right)
\]
is an element of $G$, then
\[
        \left(\begin{array}{ll} h & 0 \\ 0 & \widetilde{h}^{-1} \end{array}\right)
\]
is in $U$. Consider the subgroups of $U$ given by
\[
        H = \bigg\{\left(\begin{array}{ll} h & 0 \\ 0 & \widetilde{h}^{-1} \end{array}\right) \;\bigg|\; \left(\begin{array}{ll} h & 0 \\ 0 & I \end{array}\right) \in G\bigg\}
\]
and
\[
        N = \bigg\{\left(\begin{array}{ll} I & x \\ 0 & I \end{array}\right) \;\bigg|\; \left(\begin{array}{ll} 0 & x \\ 0 & 0 \end{array}\right) \in \frak{u}\bigg\}.
\]
Observe that $N$ is an abelian normal subgroup of $U$ and $U = N \rtimes H$ as an internal semidirect product. Furthermore, $H$ is isomorphic to a pattern subgroup of $UT_n(\mathbb{F}_q)$ under the isomorphism
\[
        \varphi : \left(\begin{array}{ll} h & 0 \\ 0 & \widetilde{h}^{-1} \end{array}\right) \mapsto h.
\]
Define $\frak{h} = f(H)$ and $\frak{n} = f(N)$. Note that the irreducible characters of $N$ are exactly the elements of the form $\psi_\lambda = \theta \circ \lambda\circ f$, where $\lambda \in \frak{n}^*$. Let
\[
        \bar{\frak{n}} = \{x \in \frak{g} \mid x_{ij} = 0 \text{ if } i > n \text{ or } j \leq n\}.
\]
Given $\lambda \in  \frak{n}^*$, define $\bar{\lambda} \in \bar{\frak{n}}^*$ by $\bar{\lambda}(x) = \frac{1}{2}\lambda(x-x^\dagger)$. Note that $\bar{\lambda}|_{\frak{n}} = \lambda$.

\bigbreak

Let $\mathcal{L}$ be the lattice of subgroups of $H$ that map to right ideal subgroups under the isomorphism $\varphi$. Each $A \in \mathcal{L}$ is equipped with a supercharacter theory with supercharacters
\[
        \{\chi \circ \varphi \mid \chi \text{ is a right ideal supercharacter of } \varphi(A)\}.
\]
It will be useful to describe these supercharacters in terms of the functionals $\mu \in \frak{a}^*$, where $\frak{a} = f(A)$. Let $M$ be the subgroup of $G$ defined by
\[
        M = \bigg\{\left(\begin{array}{ll} I & 0 \\ 0 & h\end{array}\right)\in G \bigg\},
\]
and let
\[
       R_\lambda = f^{-1}(\frak{r}_\lambda),\quad \text{where} \quad \frak{r}_\lambda = \{x \in \frak{a} \mid (m\cdot \lambda)(x) = \lambda(x) \text{ for all } m \in M\} \\
        .
\]
Then the supercharacters of $A$ are the functions
\[
        \chi_\lambda = \text{Ind}_{R_\lambda}^A(\theta \circ \lambda \circ f).
\]
(We should mention at this point that the supercharacters of algebra groups can be constructed using any Springer morphism $f$ in place of the usual $g \mapsto g-1$ map. The choice of Springer morphism has no effect on the resulting supercharacter theory.)
\bigbreak

For $\psi_\lambda \in \text{Irr}(N)$, let
\[
        H_{\psi_\lambda} = \{h \in H \mid h\bar{\lambda} = \bar{\lambda}\}.
\]
Note that $\varphi(H_{\psi_\lambda})$ is a right ideal subgroup of $UT_n(\mathbb{F}_q)$, and for $h^{-1} \in H_{\psi_\lambda}$ and $x \in \frak{n}$,
\begin{align*}
        (h^{-1} \cdot \psi_\lambda)(f^{-1}(x))
        & = \theta(\lambda(hxh^{-1})) \\
        & = \theta(\bar{\lambda}(hxh^{-1})) \\
        & = \theta(\bar{\lambda}(xh^{-1})) \\
        & = \theta(\bar{\lambda}(h(-x^\dagger))) \\
        & = \theta(\bar{\lambda}(-x^\dagger)) \\
        & = \psi_\lambda(f^{-1}(x)).
\end{align*}
It follows that $H_{\psi_\lambda} \subseteq I_\mathcal{L}(\psi_\lambda)$. Furthermore, if $k \in I_\mathcal{L}(\psi_\lambda)$, then $k \bar{\lambda}k^{-1} = \bar{\lambda}$, hence $k^{-1}\bar{\lambda} = \bar{\lambda}k^{-1}$. This means that, for $h \in H_{\psi_\lambda}$,
\[
        khk^{-1}\bar{\lambda} = kh\bar{\lambda}k^{-1} = \bar{\lambda},
\]
and $H_{\psi_\lambda}$ is in fact normal in $I_\mathcal{L}(\psi_\lambda)$, satifying condition (H1). If $\mu = k \lambda k^{-1}$ for some $k \in H$, then
\[
       \bar{\mu} = k \bar{\lambda} k^{-1},
\]
from which it follows that
\[
        H_{(\psi_\lambda)^k} = H_{\psi_{(\lambda^k)}} = H_{\psi_\lambda}^k,
\]
and condition (H2) holds. Condition (H3) is clear, and (H4) follows from the fact that $\psi_{\lambda+\mu} = \psi_\lambda\psi_\mu$.

\begin{thm}\label{sctofumlg} The above choices of $H_\psi$ produce the same supercharacter theory of $U$ as in Theorem~\ref{sctofu}.
\end{thm}

\begin{proof} Let $\eta \in \frak{u}^*$ be a functional and let $\lambda = \eta|_{\frak{n}}$. Define $\frak{h}_{\psi_\lambda} = f(H_{\psi_\lambda})$ and let $\mu=\eta|_{\frak{h}_{\psi_\lambda}}$. We claim that
\[
        \chi_\eta = \psi_\lambda \rtimes \chi_\mu.
\]

We have that
\begin{align*}
        \psi_\lambda \rtimes \chi_\mu
        & = \psi_\lambda \rtimes \text{Ind}_{R_\mu}^{H_{\psi_\lambda}} \text{Res}_{R_\mu}^{H_{\psi_\lambda}} (\theta \circ \mu \circ f) \\
        & = \psi_\lambda \rtimes \text{Res}_{R_\mu}^{H_{\psi_\lambda}} (\theta \circ \mu \circ f) \\
        & = \text{Ind}_{NR_\mu}^U\text{Res}_{NR_\mu}^{NH_{\psi_\lambda}} (\theta \circ (\lambda\oplus\mu) \circ f) \\
        & = \text{Ind}_{NR_\mu}^U\text{Res}_{NR_\mu}^{U} (\theta \circ \eta \circ f).
\end{align*}
Note that
\[
        \frak{n} \oplus \frak{r}_\mu = \{x \in \frak{u} \mid (k \cdot \eta)(x) = \eta(x) \text{ for all } k \in K\},
\]
hence $\frak{n} \oplus \frak{r}_\mu = \frak{u}_\eta$, where $\frak{u}_\eta = f(U_\eta)$ is as defined in \cite[Section 6]{andrews1}. This means that $NR_\mu = U_\eta$, and
\[
        \psi_\lambda \rtimes \chi_\mu = \text{Ind}_{U_\eta}^U\text{Res}_{U_\lambda}^U(\theta \circ \eta \circ f) = \chi_\eta.
\]
\end{proof}

The motivations for this construction are the unipotent orthogonal, symplectic and unitary groups in even dimension. Each of these groups is defined by an anti-involution with the required properties.

\subsection{Unipotent orthogonal groups}

Let
\[
        U = UO_{2n}(\mathbb{F}_q) = \{ g \in UT_{2n}(\mathbb{F}_q) \mid g^{-1} = J g^t J\}
\]
be the unipotent orthogonal group, where $q$ is a power of an odd prime. The group $U$ is defined by the anti-involution $x \mapsto J x^t J$ of $\mathfrak{ut}_{2n}(\mathbb{F}_q)$. It follows that there is a supercharacter theory of $U$ as described in Theorem~\ref{sctofumlg}, and furthermore that this supercharacter theory coincides with that produced in \cite[Section 6]{andrews1} and with the earlier construction of Andr\'e--Neto in \cite{MR2264135,MR2537684}. The semidirect product decomposition of $U$ as $U = N \rtimes H$ has
\[
        H = \bigg\{\left(\begin{array}{ll} h & 0 \\ 0 & J h^{-t}J \end{array}\right) \;\bigg|\; h \in UT_n(\mathbb{F}_q)\bigg\}
\]
and
\[
        N = \bigg\{\left(\begin{array}{ll} I & x \\ 0 & I \end{array}\right) \;\bigg|\; x \in \frak{o}_n(\mathbb{F}_q)\bigg\},
\]
where $\frak{o}_n(\mathbb{F}_q) = \{x \in \frak{gl}_n(\mathbb{F}_q) \mid Jx^tJ = -x\}$. This means that $U$ is isomorphic to the external semidirect product $UT_n(\mathbb{F}_q) \rtimes \frak{o}_n(\mathbb{F}_q)$, where $\frak{o}_n(\mathbb{F}_q)$ is considered as an additive group and the left action of $UT_n(\mathbb{F}_q)$ on $\frak{o}_n(\mathbb{F}_q)$ is given by $h \cdot x = hxJh^tJ$.

\subsection{Unipotent symplectic groups}

Let
\[
        U = USp_{2n}(\mathbb{F}_q) = \{ g \in UT_{2n}(\mathbb{F}_q) \mid g^{-1} = - \Omega g^t \Omega\}
\]
be the unipotent symplectic group, where
\[
        \Omega = \left(\begin{array}{ll} 0 & J \\ -J & 0 \end{array}\right).
\]
The group $U$ is defined by the anti-involution $x \mapsto -\Omega x^t \Omega$ of $\mathfrak{ut}_{2n}(\mathbb{F}_q)$. It follows that there is a supercharacter theory of $U$ as described in Theorem~\ref{sctofumlg}, and furthermore that this supercharacter theory coincides with that produced in \cite[Section 6]{andrews1} and with the earlier construction of Andr\'e--Neto in \cite{MR2264135,MR2537684}. The semidirect product decomposition of $U$ as $U = N \rtimes H$ has
\[
        H = \bigg\{\left(\begin{array}{ll} h & 0 \\ 0 & J h^{-t}J \end{array}\right) \;\bigg|\; h \in UT_n(\mathbb{F}_q)\bigg\}
\]
and
\[
        N = \bigg\{\left(\begin{array}{ll} I & x \\ 0 & I \end{array}\right) \;\bigg|\; x \in \frak{o}_n^\perp(\mathbb{F}_q)\bigg\},
\]
where $\frak{o}_n^\perp(\mathbb{F}_q) = \{x \in \frak{gl}_n(\mathbb{F}_q) \mid Jx^tJ = x\}$. This means that $U$ is isomorphic to the external semidirect product $UT_n(\mathbb{F}_q) \rtimes \frak{o}_n^\perp(\mathbb{F}_q)$, where $\frak{o}_n^\perp(\mathbb{F}_q)$ is considered as an additive group and the left action of $UT_n(\mathbb{F}_q)$ on $\frak{o}_n^\perp(\mathbb{F}_q)$ is given by $h \cdot x = hxJh^tJ$.

\subsection{Unipotent unitary groups}

Let
\[
        U = UU_{2n}(\mathbb{F}_{q^2}) = \{ g \in UT_{2n}(\mathbb{F}_{q^2}) \mid g^{-1} = J \overline{g}^t J\}
\]
be the unipotent unitary group, where $\overline{g}_{ij} = (g_{ij})^q$ and $q$ is a power of an odd prime. The group $U$ is defined by the anti-involution $x \mapsto J\overline{x}^tJ$ of $\frak{ut}_{2n}(\mathbb{F}_{q^2})$ (considered as an $\mathbb{F}_q$-algebra). It follows that there is a supercharacter theory of $U$ as described in Theorem~\ref{sctofumlg}, and furthermore that this supercharacter theory coincides with that produced in \cite[Section 6]{andrews1}. The semidirect product decomposition of $U$ as $U = N \rtimes H$ has
\[
        H = \bigg\{\left(\begin{array}{ll} h & 0 \\ 0 & J \overline{h}^{-t}J \end{array}\right) \;\bigg|\; h \in UT_n(\mathbb{F}_{q^2})\bigg\}
\]
and
\[
        N = \bigg\{\left(\begin{array}{ll} I & x \\ 0 & I \end{array}\right) \;\bigg|\; x \in \frak{u}_n(\mathbb{F}_{q^2})\bigg\},
\]
where $\frak{u}_n(\mathbb{F}_{q^2}) = \{x \in \frak{gl}_n(\mathbb{F}_{q^2}) \mid J\overline{x}^tJ = -x\}$. This means that $U$ is isomorphic to the external semidirect product $UT_n(\mathbb{F}_{q^2}) \rtimes \frak{u}_n(\mathbb{F}_{q^2})$, where $ \frak{u}_n(\mathbb{F}_{q^2})$ is considered as an additive group and the left action of $UT_n(\mathbb{F}_q)$ on $ \frak{u}_n(\mathbb{F}_{q^2})$ is given by $h \cdot x = hxJ\overline{h}^tJ$.

\section{Coarser supercharacter theories of $UT_n(\mathbb{F}_q)$}\label{seccoarsersct}

One advantage of considering supercharacter theories in terms of the method of little groups is that it is often easy to modify supercharacter theories to obtain coarser or finer supercharacter theories. Suppose $G = N \rtimes H$ with $N$ abelian and that $\mathcal{L}'\subseteq \mathcal{L}$ are two sublattices of the subgroup lattice of $H$ satisfying (L1)--(L4). Furthermore, suppose that
\begin{enumerate}
\item the supercharacter theory assigned to each $K \in \mathcal{L}'$ is (not necessarily strictly) coarser than the supercharacter theory assigned to $K$ as an element of $\mathcal{L}$; and
\item for each $\psi \in \text{Irr}(N)$, the subgroups $H_\psi'\in \mathcal{L}'$ and $H_\psi\in \mathcal{L}$ have $H_\psi'\subseteq H_\psi$.
\end{enumerate}

\begin{lem}\label{coarsersct} The supercharacter theory of $G$ corresponding to the lattice $\mathcal{L}'$ and the subgroups $H_\psi'$ is (not necessarily strictly) coarser than the the supercharacter theory of $G$ corresponding to the lattice $\mathcal{L}$ and the subgroups $H_\psi$.
\end{lem}

\begin{proof} Let $\psi \in \text{Irr}(N)$ and let $\chi$ be a supercharacter of $H_\psi'$. Then $\psi \rtimes \chi = \psi \rtimes \text{Ind}_{H_\psi'}^{H_\psi}(\chi)$; as $\text{Ind}_{H_\psi'}^{H_\psi}(\chi)$ is a superclass function of $H_\psi$, $\psi \rtimes \chi$ is a superclass function of $G$ in the supercharacter theory of $G$ corresponding to the lattice $\mathcal{L}$ and the subgroups $H_\psi$. It follows that the supercharacter theory of $G$ corresponding to the lattice $\mathcal{L}'$ and the subgroups $H_\psi'$ is at least as coarse as the supercharacter theory of $G$ corresponding to the lattice $\mathcal{L}$ and the subgroups $H_\psi$.
\end{proof}
\noindent
\textbf{Cautionary Example.} Let $G = N \rtimes H$ be any semidirect product with $N$ abelian; we can choose
\begin{enumerate}
\item $\mathcal{L} = \{\text{Normal subgroups of }H\}$ under their usual character theory, and
\item $\mathcal{L}' = \{\text{Normal subgroups of }H\}$ under the supercharacter theories formed from conjugation action of $H$.
\end{enumerate}
As long as $H_\psi = H_\psi'$ for all $\psi \in \text{Irr}(N)$, these two choices of lattice will give the same supercharacter theory of $G$, even though in general the supercharacter theories of the elements of $\mathcal{L}'$ can be strictly coarser than those in $\mathcal{L}$.

\bigbreak

As an example, we construct a collection of supercharacter theories of $UT_n(\mathbb{F}_q)$ which are coarser than the usual supercharacter theory and have nice indexing sets for the superclasses and supercharacters. If $G = 1+\frak{g}$ is an algebra group and $K = 1+\frak{k}$ is a subgroup with $\frak{k}$ a two-sided ideal of $\frak{g}$, there is a supercharacter theory of $K$ with superclasses
\[
        K_g = \{h \in K \mid f(h) \in Gf(g)G\}
\]
and supercharacters
\[
        \chi_\lambda = \sum_{\mu \in G \lambda G} \theta \circ \mu \circ f.
\]
This result is a consequence of Theorem~\ref{alggpsct}; the supercharacter theory will be referred to as the \emph{supernormal supercharacter theory} of $K$ (see \cite{MR2774691} for more on supernormal subgroups).

\bigbreak

To construct a supercharacter theory of $UT_n(\mathbb{F}_q)$, let $n = m+k$, and let $H = H_k \times H_m$ and $N$ be as in Section~\ref{typea}. Let $\mathcal{L}$ be the lattice of subgroups of $H$ of the form $K_k \times K_m$ with $K_k = 1+\frak{k}_k$, $K_m = 1+\frak{k}_m$, and $\frak{k}_k$ and $\frak{k}_m$ two-sided ideals of $\frak{ut}_k(\mathbb{F}_q)$ and $\frak{ut}_m(\mathbb{F}_q)$, respectively. We equip such a subgroup $K_k \times K_m$ with the direct product of the supernormal supercharacter theories of $K_k$ and $K_m$. For each $\psi \in \text{Irr}(N)$, let $H_\psi = I_\mathcal{L}(\psi)$. Observe that
\begin{enumerate}
\item this choice of lattice $\mathcal{L}$ and of $H_\psi$ satisfies (L1)--(L4) and (H1)--(H4), and
\item by Lemma~\ref{coarsersct} the resulting supercharacter theory is coarser than the usual supercharacter theory.
\end{enumerate}
We will refer to this supercharacter theory as $SCT(n,k)$.

\bigbreak

If $\frak{n} = f(N)$, the irreducible characters of $N$ are exactly the elements $\psi_\lambda = \theta \circ \lambda \circ f$, where $\lambda \in \frak{n}^*$. We have that
\[
        H_{\psi_\lambda} = \bigg\{h \in H \;\bigg|\;\begin{array}{l} h_{ij} = 0 \text{ if there exist } i' \text{ and } j' \text{ with } \\ \lambda(e_{i'j'}) \neq 0 \text{ and } i' \leq i<k \text{ or } k<j \leq j'\end{array}\bigg\}.
\]

\begin{example} Let $n=7$, $k=3$, and $m=4$, and let $G = UT_7(\mathbb{F}_q)$. Define $\lambda \in \frak{n}^*$ by
\[
        \lambda(x) = x_{15}+x_{36}.
\]
In other words, $\lambda(x)$ is determined by the entries denoted by $\circ$ below:
\[
        \left(\begin{array}{lll|llll}
        1 & 0 & 0 & * & \circ & * & * \\
        0 & 1 & 0 & * & * & * & * \\
        0 & 0 & 1 & * & * & \circ & * \\ \hline
        0 & 0 & 0 & 1 & 0 & 0 & 0 \\
        0 & 0 & 0 & 0 & 1 & 0 & 0 \\
        0 & 0 & 0 & 0 & 0 & 1 & 0 \\
        0 & 0 & 0 & 0 & 0 & 0 & 1
        \end{array}\right).
\]
For this choice of $\lambda$, $H_{\psi_\lambda}$ will be the subgroup
\[
H_{\psi_\lambda}  = \left\{\left(\begin{array}{lll|llll}
        1 & 0 & 0 & 0 & \circ & 0 & 0 \\
        0 & 1 & 0 & 0 & 0 & 0 & 0 \\
        0 & 0 & 1 & 0 & 0 & \circ & 0 \\ \hline
        0 & 0 & 0 & 1 & 0 & 0 & * \\
        0 & 0 & 0 & 0 & 1 & 0 & * \\
        0 & 0 & 0 & 0 & 0 & 1 & * \\
        0 & 0 & 0 & 0 & 0 & 0 & 1
        \end{array}\right) \in H\right\},
\]
where once again the entries labeled with $\circ$ are the entries that determine $\lambda$. For comparison, note that in the usual supercharacter theory of $UT_7(\mathbb{F}_q)$, the subgroup $H_{\psi_\lambda}$ consists of the elements
\[
H_{\psi_\lambda}  = \left\{\left(\begin{array}{lll|llll}
        1 & 0 & 0 & 0 & \circ & 0 & 0 \\
        0 & 1 & * & 0 & 0 & 0 & 0 \\
        0 & 0 & 1 & 0 & 0 & \circ & 0 \\ \hline
        0 & 0 & 0 & 1 & 0 & 0 & * \\
        0 & 0 & 0 & 0 & 1 & 0 & * \\
        0 & 0 & 0 & 0 & 0 & 1 & * \\
        0 & 0 & 0 & 0 & 0 & 0 & 1
        \end{array}\right) \in H\right\}.
\]
\end{example}

The supercharacters and superclasses of $UT_n(\mathbb{F}_q)$ under the algebra group supercharacter theory are indexed by $\mathbb{F}_q$-set partitions (see \cite[Section~2]{MR2592079} for details). To each $\mathbb{F}_q$-set partition $\eta$, we assign representatives $g_\eta \in UT_n(\mathbb{F}_q)$ and $\lambda_\eta \in \frak{ut}_n(\mathbb{F}_q)^*$, and define supercharacters and superclasses corresponding to $\eta$. Let
\[
 K_\eta = K_{g_\eta} \quad \text{and}\quad \chi_\eta = \frac{|UT_n(\mathbb{F}_q)\lambda UT_n(\mathbb{F}_q)|}{|UT_n(\mathbb{F}_q)\lambda|}\chi_{\lambda_\eta}.
\]
where $K_{g_\eta}$ and $\chi_{\lambda_\eta}$ are as in Theorem~\ref{alggpsct}. The pair $(\{K_\eta\},\{\chi_\eta\})$ gives the same supercharacter theory of $UT_n(\mathbb{F}_q)$ as in Theorem~\ref{alggpsct}, and has the advantage that
\[
        \chi_\eta = \sum_{\mu \in UT_n(\mathbb{F}_q)\lambda_\eta UT_n(\mathbb{F}_q)} \theta \circ \mu \circ f
\]
and
\[
        \chi_\eta = \sum \psi(1)\psi,
\]
where the sum is over some set of irreducible characters $\psi$.

\bigbreak

We can describe the supercharacters and superclasses of $SCT(n,k)$ in a similar manner to those of the algebra group supercharacter theory. If an $\mathbb{F}_q$-set partition $\eta$ has the property that if $i\overset{a}{\frown} j$ is an arc of $\eta$ with $i \leq k <j$, then there are no arcs in $\lambda$ of the form $i' \overset{a'}{\frown} j'$ with $i <i'<j'\leq k$ or $k<i'<j'<j$, we will call $\eta$ a $k$\emph{-nonnesting} $\mathbb{F}_q$\emph{-set partition}.

\bigbreak

Given an $\mathbb{F}_q$-set partition $\eta$, let
\begin{align*}
        \overline{\eta}_k &= \bigg\{i \overset{a}{\frown}j \in \eta \;\bigg|\; \begin{array}{l} \text{ there are no } i'\overset{a'}{\frown}j' \in \eta \text{ with } \\ i<i'<j'\leq k<j   \text{ or } i \leq k<i'<j'<j \end{array} \bigg\} \quad\text{and}\\
        \widetilde{\eta}_k &= \bigg\{i \overset{a}{\frown}j \in \eta \;\bigg|\; \begin{array}{l} \text{ there are no } i'\overset{a'}{\frown}j' \in \eta \text{ with } \\ i'<i<j\leq k<j'   \text{ or } i' \leq k<i<j<j' \end{array} \bigg\}.
\end{align*}
In other words, we obtain $\overline{\eta}_k$ and $\widetilde{\eta}_k$ from $\eta$ by two different methods of removing arcs that prevent $\eta$ from being a $k$-nonnesting $\mathbb{F}_q$-set partition.

\begin{example} Let $n=12$, $k = 5$, and $m = 7$, and
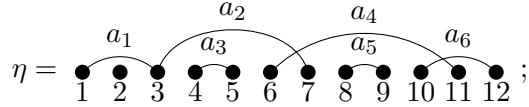
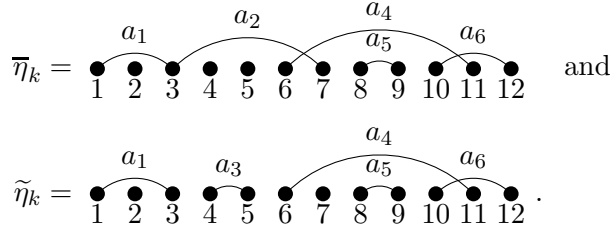
\[
        \eta = \begin{tikzpicture}[baseline={([yshift=-2ex]current bounding box.center)}]
	\fill (0,0) circle (.1) node[below]{1};
    \fill (.5,0) circle (.1) node[below]{2};
    \fill (1,0) circle (.1) node[below]{3};
    \fill (1.5,0) circle (.1) node[below]{4};
    \fill (2,0) circle (.1) node[below]{5};
    \fill (2.5,0) circle (.1) node[below]{6};
    \fill (3,0) circle (.1) node[below]{7};
    \fill (3.5,0) circle (.1) node[below]{8};
    \fill (4,0) circle (.1) node[below]{9};
    \fill (4.5,0) circle (.1) node[below]{10};
    \fill (5,0) circle (.1) node[below]{11};
    \fill (5.5,0) circle (.1) node[below]{12};
    \draw (0,0) to [out=45, in=135] node[above] {$a_1$} (1,0);
	\draw (1,0) to [out=75, in=105] node[above] {$a_2$} (3,0);
    \draw (1.5,0) to [out=45, in=135] node[above] {$a_3$} (2,0);
	\draw (2.5,0) to [out=45, in=135] node[above] {$a_4$} (5,0);
    \draw (3.5,0) to [out=45, in=135] node[above] {$a_5$} (4,0);
	\draw (4.5,0) to [out=45, in=135] node[above] {$a_6$} (5.5,0);
    \end{tikzpicture};
\]
then we have
\begin{align*}
        \overline{\eta}_k &= \begin{tikzpicture}[baseline={([yshift=-2ex]current bounding box.center)}]
	\fill (0,0) circle (.1) node[below]{1};
    \fill (.5,0) circle (.1) node[below]{2};
    \fill (1,0) circle (.1) node[below]{3};
    \fill (1.5,0) circle (.1) node[below]{4};
    \fill (2,0) circle (.1) node[below]{5};
    \fill (2.5,0) circle (.1) node[below]{6};
    \fill (3,0) circle (.1) node[below]{7};
    \fill (3.5,0) circle (.1) node[below]{8};
    \fill (4,0) circle (.1) node[below]{9};
    \fill (4.5,0) circle (.1) node[below]{10};
    \fill (5,0) circle (.1) node[below]{11};
    \fill (5.5,0) circle (.1) node[below]{12};
    \draw (0,0) to [out=45, in=135] node[above] {$a_1$} (1,0);
	\draw (1,0) to [out=45, in=135] node[above] {$a_2$} (3,0);
	\draw (2.5,0) to [out=45, in=135] node[above] {$a_4$} (5,0);
    \draw (3.5,0) to [out=45, in=135] node[above] {$a_5$} (4,0);
	\draw (4.5,0) to [out=45, in=135] node[above] {$a_6$} (5.5,0);
    \end{tikzpicture} \quad\text{and} \\
       \widetilde{\eta}_k &= \begin{tikzpicture}[baseline={([yshift=-2ex]current bounding box.center)}]
	\fill (0,0) circle (.1) node[below]{1};
    \fill (.5,0) circle (.1) node[below]{2};
    \fill (1,0) circle (.1) node[below]{3};
    \fill (1.5,0) circle (.1) node[below]{4};
    \fill (2,0) circle (.1) node[below]{5};
    \fill (2.5,0) circle (.1) node[below]{6};
    \fill (3,0) circle (.1) node[below]{7};
    \fill (3.5,0) circle (.1) node[below]{8};
    \fill (4,0) circle (.1) node[below]{9};
    \fill (4.5,0) circle (.1) node[below]{10};
    \fill (5,0) circle (.1) node[below]{11};
    \fill (5.5,0) circle (.1) node[below]{12};
    \draw (0,0) to [out=45, in=135] node[above] {$a_1$} (1,0);
    \draw (1.5,0) to [out=45, in=135] node[above] {$a_3$} (2,0);
    \draw (2.5,0) to [out=45, in=135] node[above] {$a_4$} (5,0);
    \draw (3.5,0) to [out=45, in=135] node[above] {$a_5$} (4,0);
	\draw (4.5,0) to [out=45, in=135] node[above] {$a_6$} (5.5,0);
    \end{tikzpicture}.
\end{align*}
\end{example}

The constructions $\overline{\eta}_k$ and $\widetilde{\eta}_k$ each define an equivalence relation on the set of $\mathbb{F}_q$-set partitions. In both cases, the equivalence classes are indexed by the $k$-nonnesting $\mathbb{F}_q$-set partitions. It follows that the two relations have the same number of equivalence classes.

\bigbreak

If $K_\eta$ denotes the superclass of $UT_n(\mathbb{F}_q)$ associated to the $\mathbb{F}_q$-set partition $\eta$ in the algebra group supercharacter theory, let
\[
        K_{[\eta]_k} = \bigcup_{\widetilde{\nu}_k = \widetilde{\eta}_k} K_{\nu}.
\]
Similarly, if $\chi_\eta$ is the supercharacter associated to $\eta$ in the algebra group supercharacter theory as defined above, let
\[
        \chi_{[\eta]_k} = \sum_{\overline{\nu}_k=\overline{\eta}_k} \chi_{\nu}.
\]

\begin{thm}\label{sctnk} The characters $\chi_{[\eta]_k}$, along with the subsets $K_{[\eta]_k}$, are the supercharacters and superclasses of $SCT(n,k)$. If $\eta$ and $\nu$ are $\mathbb{F}_q$-set partitions and $g \in K_{[\nu]_k}$, we have that
\[
        \chi_{[\eta]_k}(g) = \left\{\begin{array}{ll}
        \frac{\chi_{[\eta]_k}(1)}{\chi_{\eta}(1)}\chi_{\eta}(g)  & \quad\text{if there are no } i \overset{a}{\frown} j \in \eta \text{ and } i' \overset{a'}{\frown} j' \in \nu \\
        {} & \quad\text{with }i<i'<j'\leq k<j \text{ or } i \leq k<i'<j'<j, \\
        {}&{}\\
        0 &\quad\text{otherwise.}\end{array}\right.
\]
\end{thm}

\begin{proof} Let $G = UT_n(\mathbb{F}_q)$ and $\frak{g} = \frak{ut}_n(\mathbb{F}_q)$. Choose $\alpha \in \frak{g}^*$, and let $\lambda = \alpha|_{\frak{n}}$. We have that
\[
        H_{\psi_\lambda} = \bigg\{h \in H \;\bigg|\;\begin{array}{l} h_{ij} = 0 \text{ if there exist } i' \text{ and } j' \text{ with } \\ \lambda(e_{i'j'}) \neq 0 \text{ and } i' \leq i<k \text{ or } k<j \leq j'\end{array}\bigg\},
\]
where $\psi_\lambda = \theta \circ \lambda \circ f$. Define $\frak{h}_{\psi_\lambda} = f(H_{\psi_\lambda})$, and let $\mu = \alpha|_{\frak{h}_{\psi_\lambda}}$. The supercharacter of $SCT(n,k)$ associated to $\lambda$ and $\mu$ is
\begin{align*}
        \psi_\lambda \rtimes \sum_{\beta \in H_{\psi_\lambda}\mu H_{\psi_\lambda}}\theta \circ \beta \circ f
        & = \text{Ind}_{NH_{\psi_\lambda}}^G\bigg(\sum_{\beta \in S}\theta \circ \beta \circ f \bigg) \\
        & = \sum_{\substack{\beta \in \frak{g}^* \\ \beta|_{\frak{n}\oplus\frak{h}_{\psi_\lambda}} \in S}}\theta \circ \beta \circ f,
\end{align*}
where $S = NH_{\psi_\lambda}(\lambda\oplus\mu) NH_{\psi_\lambda}$, by \cite[Lemma~4.5]{andrews1}. Note that if $\beta \in \frak{g}^*$ and $\theta \circ \beta \circ f$ is a constituent of $\chi_\eta$, then $\beta \in \{\gamma \in \frak{g}^* \mid \gamma|_{\frak{n}\oplus\frak{h}_{\psi_\lambda}} \in NH_{\psi_\lambda}(\lambda\oplus\mu) NH_{\psi_\lambda}\}$ if and only if $\overline{\eta}_k = \overline{\nu}_k$, where $\theta \circ \alpha \circ f$ is a constituent of $\chi_\nu$. It follows that the supercharacters of $SCT(n,k)$ are $\{\chi_{[\eta]_k}\}$.

\bigbreak

There are no $i \overset{a}{\frown} j \in \eta$ and $i' \overset{a'}{\frown} j' \in \nu$ with $i<i'<j'\leq k<j \text{ or } i \leq k<i'<j'<j$ if and only if $g \in NH_{\psi_\lambda}$. Furthermore, $\chi_{[\eta]_k}$ is the character of $G$ obtained by inducing the supernormal supercharacter of $NH_{\psi_\lambda}$ associated to $\lambda \oplus \mu$ to $G$. It follows that
\[
        \chi_{[\eta]_k}(g) = \left\{\begin{array}{ll}
        \frac{\chi_{[\eta]_k}(1)}{\chi_{\eta}(1)}\chi_{\eta}(g)  & \quad\text{if there are no } i \overset{a}{\frown} j \in \eta \text{ and } i' \overset{a'}{\frown} j' \in \nu \\
        {} & \quad\text{with }i<i'<j'\leq k<j \text{ or } i \leq k<i'<j'<j, \\
        {}&{}\\
        0 &\quad\text{otherwise,}\end{array}\right.
\]
and the superclasses of $SCT(n,k)$ are the sets $K_{[\nu]_k}$.
\end{proof}

Observe that
\begin{enumerate}
\item $SCT(n,0)=SCT(n,n)$, and this supercharacter theory is just the usual algebra group supercharacter theory,
\item $SCT(n,k)$ is strictly coarser than the usual algebra group supercharacter theory for all $1 \leq k \leq n-1$, and
\item If $1 \leq k < k' \leq n-1$, then $SCT(n,k)$ and $SCT(n,k')$ are incomparable.
\end{enumerate}

If $S \subseteq [n]$, we define $SCT(n,S)$ to be the lattice-theoretic join of the supercharacter theories $SCT(n,k)$ for all $k \in S$. The following results follow from Lemma~\ref{lemscljoin}.
\begin{enumerate}
\item If $S \subsetneq T$, then $SCT(n,S)$ is strictly finer than $SCT(n,T)$.
\item If $S$ and $T$ are incomparable, then so are $SCT(n,S)$ and $SCT(n,T)$.
\item The supercharacters and superclasses of $SCT(n,[n])$ are indexed by the nonnesting $\mathbb{F}_q$-set partitions.
\end{enumerate}
The supercharacter theory $SCT(n,[n])$ is investigated in more depth by the author in \cite{andrews3}.

\section{Further directions}

Unfortunately, the method of little groups cannot be applied to the groups $UO_{2n+1}(\mathbb{F}_q)$ or $UU_{2n+1}(\mathbb{F}_{q^2})$ as there is no semidirect product decomposition with an abelian normal subgroup. There is some hope that a modification of the method could work, however; $UO_{2n+1}(\mathbb{F}_q)$ has a normal subgroup $N \cong \frak{o}_n(\mathbb{F}_q)$ with $G/N \cong UT_{n+1}(\mathbb{F}_q)$ (a similar statement can be made for $UU_{2n+1}(\mathbb{F}_{q^2})$). In these cases it is not in general true that each $\psi \in \text{Irr}(N)$ can be extended to a character of $I_G(\psi)$. Hopefully the construction can be modified to a group $G$ with abelian normal subgroup $N$ even when $N$ is not the normal complement of any subgroup of $G$ and irreducible characters of $N$ cannot be extended to their inertial subgroups. If so, it seems that the supercharacter theories of the unipotent orthogonal and unitary groups in odd dimension could be recovered from right ideal supercharacter theories.

\bigbreak

It would be also of interest to investigate finer supercharacter theories in the orthogonal, symplectic, and unitary types. Note that the choices of $H_\psi$ in Theorem~\ref{sctofumlg} are not always maximal, and as such there is a finer supercharacter theory that can be constructed from the same lattice. Furthermore, these supercharacters will each be of the form
\[
        \chi = c\sum_{\lambda \in S} \theta \circ \lambda \circ f
\]
for some constant $c$ and some subset $S \subseteq \frak{u}^*$.

\bigbreak

The supercharacter theory $SCT(n,[n])$ described in Section~\ref{seccoarsersct} has supercharacters and superclasses indexed by the nonnesting $\mathbb{F}_q$-set partitions. In \cite{andrews3}, the author calculates the supercharacter table of this supercharacter theory and generalizes it to arbitrary pattern groups. This leads to the construction of a Hopf monoid on the nonnesting supercharacter theories that is analogous to the one constructed in \cite{MR3117506}.

\section{Acknowledgements}

I would like to thank Nat Thiem for his numerous helpful suggestions and insights.

\bibliography{bibfile}
	\bibliographystyle{plain}

\end{document}